%% file: planarcagesNGF20181116.tex
\documentclass{amsart}

\usepackage{amsmath,amsthm,amsfonts}
\usepackage[utf8]{inputenc}
\usepackage{parskip}

\usepackage{enumerate}
\usepackage{lineno}
\usepackage[disable]{todonotes}
\usepackage{datetime2}
\newtheorem{theorem}{Theorem}
\newtheorem{lemma}[theorem]{Lemma}
\newtheorem{corollary}[theorem]{Corollary}
\newtheorem{proposition}[theorem]{Proposition}

\newtheorem*{lemma*}{Lemma}

\newcommand{\cage}[1][k]{\ensuremath{(#1,g)}-cage}

\newcommand{\bicageT}[1][r,m]{\ensuremath{(\{#1\};3)}-cage}
\newcommand{\bicagesT}[1][r,m]{\ensuremath{(\{#1\};3)}-cages}
\newcommand{\npbicageT}[1][r,m]{\ensuremath{n_p(\{#1\};3)}}
\newcommand{\bigraphT}[1][r,m]{\ensuremath{(\{#1\};3)}-graph}
\newcommand{\bigraphsT}[1][r,m]{\ensuremath{(\{#1\};3)}-graphs}

\newcommand{\bicages}[2][r,m]{\ensuremath{(\{#1\};#2)}-cages}

\newcommand{\bigraphs}[2][r,m]{\ensuremath{(\{#1\};#2)}-graphs}

\title{Regular and biregular planar cages}

\author[G.Araujo-Pardo]{Gabriela Araujo-Pardo $^1$}
\address{$^1$ Instituto de Matemáticas, UNAM}
\email{ garaujo@math.unam.mx}

\author[F. Barrera-Cruz]{Fidel Barrera-Cruz $^2$}
\address{$^2$ Sunnyvale, CA}
\email{ fidel.barrera@gmail.com}

\author[N. Garc\'{i}a-Col\'{i}n]{Natalia Garc\'{i}a-Col\'{i}n $^3$}
\address{$^3$ CONACYT Research Fellow - INFOTEC Centro de
  Investigación en Tecnologías de la Información y Comunicación,
  Mexico. Corresponding Author.} 
  \email{  natalia.garcia@infotec.mx}

\begin{document}

\begin{abstract}
  We study the \textit{Cage Problem} for regular and biregular planar
  graphs. A $(k,g)$-\textit{graph} is a $k$-regular graph with girth $g$. A $(k,g)$-\textit{cage} is a $(k,g)$-graph of minimum order. It is not difficult to conclude that the \textit{regular
   planar cages} are the \textit{Platonic Solids}. A $(\{r,m\};g)$-\textit{graph} is a graph of girth $g$ whose vertices have
degrees $r$ and $m.$  A $(\{r,m\};g)$-\textit{cage}  is a $(\{r,m\};g)$-graph of minimum order. In this case we determine the triplets of values $(\{r,m\};g)$ for which there exist planar $(\{r,m\};g)$--graphs, for all those values we construct examples. Furthermore, for many triplets $(\{r,m\};g)$ we build the $(\{r,m\};g)$-cages.
\end{abstract}

\maketitle

\section{Introduction}
We  only consider finite {simple} \todo{Shall add ``simple'', {\color{blue} that is not necessary says "undirected" when we are working with graphs}. Would this simplify proof of Lemma~\ref{lem:chorizo2}? N:no} graphs. The
girth of a graph is the length of a smallest cycle. A
$(k,g)$-\textit{graph} is a $k$-regular graph with girth $g$. A
$(k,g)$-\textit{cage} is a $(k,g)$-graph of minimum order, $n(k,g).$
These graphs were introduced by Tutte in 1947 (see~\cite{T47}). The
\textit{Cage Problem} consists of finding the $(k,g)$-cages for any pair
integers $k\geq 2$ and $g\geq 3.$ However, this challenge has proven
to be very difficult even though the existence of $(k,g)$-graphs was
proved by Erd\"os and Sachs in 1963 (see~\cite{ES63}).

There is a known natural lower bound for the order of a cage, called
\textit{Moore's lower bound} and denoted by $n_0(r,g).$ It is obtained
by counting the vertices of a rooted tree, ${{T}}_{(g-1)/2}$ with
radius $(g-1)/2$, if $g$ is odd; or the vertices of a ``double-tree''
rooted at an edge (that is, two different rooted trees
${{T}}_{(g-3)/2}$ with the root vertices incident to an edge) if $g$
is even (see~\cite{CL04,EJ08}\todo{Can we add page number(s) or
  section for the textbook? N: Maybe. I don't know.}). Consequently, the challenge is to find
$(k,g)$-graphs with minimum order. In each case, the smallest known
example 
is called a \textit{record graph}. For a complete review about known
cages, record graphs, and different techniques and constructions
see~\cite{EJ08}.

As a generalization of this problem, in 1981, Chartrand, Would and
Kapoor introduced in the concept of \textit{biregular
  cage}~\cite{CH81}. A biregular $(\{r,m\};g)$-cage, for
$2\leq r < m$, is a graph of girth $g\geq 3$ whose vertices have
degrees $r$ and $m$ and are of the smallest order among all such
graphs. See~\cite{AABLL13,ABLM13,AEJ16,EJ15} for \textit{record
  biregular graphs} and the bounds given by those.

Let $n_0(\{r,m\};g)$ be the order  of  an $(\{r,m\};g)$-cage. There is also a  Moore lower bound tree construction for biregular cages (see~\cite{CH81}), which gives the following bounds:

\begin{theorem}\label{thm:Dgcages} 
For $r<m$ the following bounds hold.
\begin{gather*}
n_0(\{r, m\};g) \geq 1+\sum_{i=1}^{t-1} {m(r-1)^{i}} \text{\; for } g=2t+1 \\
n_0(\{r, m\};g) \geq 1+\sum_{i=1}^{t-2} {m(r-1)^{i}}+ (r-1)^{t-1} \text{\; for } g=2t 
\end{gather*}
\end{theorem}

Also, in the same paper Chartrand, et.al. proved that:

\begin{theorem}\label{thm:Ch_2_m_cages} 

For $2<m$ the following hold:
\begin{gather*}
n(\{2, m\};g) = \frac{m(g-2)+4}{2} \text{\; for } g=2t \\
n(\{2, m\};g) = \frac{m(g-1)+2}{2} \text{\; for } g=2t+1
\end{gather*}
\end{theorem}

\begin{theorem}\label{thm:Ch_girth4} 
For $r<m$, $n(\{r, m\};4) = 	r+m.$
\end{theorem}

Finally, we would like to mention the closely related degree diameter problem; determine the largest graphs or digraphs of given maximum degree and given diameter. This problem has also been studied in the context of embeddability, namely:

Let $S$ be an arbitrary connected, closed surface (orientable or not) and let $n_{\Delta,D}(S)$ be the largest order of a graph of maximum degree at most $\Delta$ and diameter at most $D$, embeddable in S.

For a good survey on both the degree diameter problem and it's embedded version see \cite{MS2013}.

\subsection{Contribution}

In this paper we study a variation of the cage problem, for  when we want to find the minimum order planar $(r,g)$-graphs or \emph{regular planar cages} and $(\{r,m\};g)$-graphs or \emph{biregular planar cages}.

The $g$ parameter of a planar cage is a lower bound for the minimum
length of a face in the graph's embedding\todo{Do we mean facial cycle? N:yes, but we don't really need to define it. Thus the terminology used.
  Just to get isolated edges out of the way. N: I don't understand this comment. Also, shall we say for any planar (or combinatorial) embedding of the graph? We do not care for "any" but for a particular one.}. It is not
difficult to prove that the {regular planar cages} are the
\textit{Platonic Solids}. [Section~\ref{regular}]

For $(\{r,m\};g)$-graphs we denote as $n_p(\{r,m\};g)$ the order of a
$(\{r,m\};g)$-planar cage. In Section 3, we discover that the set of triads $(\{r,m\};g)$ for which planar $(\{r,m\};g)$-graphs may exist is
\begin{gather*}
\{(\{r, m\}, 3) | 2\leq r \leq 5, r <m \}  \hspace{0.5cm}
\{(\{r, m\}, 4) | 2\leq r \leq 3, r <m \}  \\
\{(\{r, m\}, 5) | 2\leq r \leq 3, r <m \}  \hspace{0.5cm} 
\{(\{2, m\}, g) | , 2 <m, 6\leq g \}.
\end{gather*}

We provide upper and lower bounds for all $n_p(\{r,m\};g)$. We construct planar $(\{r,m\};g)$--graphs for all possible triads and construct planar $(\{r,m\};g)$--cages for 
\begin{gather*}
(\{2, m\}, 3), (\{3, m\}, 3), (\{4, m\}, 3), (\{5, 6\}, 3), (\{5, 7\}, 3) \\
 (\{2, m\}, 4), (\{3, 4 \leq m \leq 13\}, 4), (\{3, m=5k-1\}, 4) \text{ with } k\geq 3, \\
(\{2, m\}, 5) \text{ and } (\{2, m\}, 6).
\end{gather*}

We also remark that, for the triplets of the form $(\{5,m\};3)$ and $(\{3,m\};4)$ for which we cannot assert that we have found a planar biregular cage, we provide constructions with a small excess (i.e. the upper bounds provided by the constructions and the lower bounds for  $n_p(\{5,m\};3)$ and $n_p(\{3,m\};4)$, respectively, differ only by a small constant).

In Section 4 we provide the proofs for some technical lemmas that we will use throughout the paper. Finally, in Section 5 we state some conclusions and further research directions.

\section{Regular planar cages}\label{regular}

The $(2,g)$-graphs are cycles on $g$ vertices. Since these graphs are
both planar and known to be cages, it follows that these graphs are
the $(2,g)$-planar cages. In this section we will investigate the
existence of $(k,g)$-planar graphs for $k\geq 3,$ and find smallest ones.

It is well known that any planar graph has at least one vertex of
degree at most $5,$ hence $3 \leq k \leq 5.$ Recall that
if $G$ is a planar $(k,g)$-graph, then its planar dual $G^*$ is also
planar, all its faces are of size $k,$ and its girth is bounded above
by $k;$ while the degree of its vertices is bounded below by $g.$
Using these we can deduce that $g \leq 5.$ Thus, $3 \leq g \leq 5, $
and $(k,g)$-planar graphs may only exist for $k, g \in \{3, 4, 5\}.$

We would like to highlight that if $G$ is a $(k,g)$-graph, then $G^*$
is not necessarily a $(g,k)$-graph.\todo{Did not think about this, do
  we have a example? N: look at the graphs that provide the lower bounds.}

Let $G$ be an embedded planar graph, then we will denote as
$v=v(G)$ its order, as $e=e(G)$ its size and its number of faces as
$f=f(G).$ We will now argue that:

\begin{theorem}
The {planar $(k,g)$-- cages}  \todo{We should be consistent with
  notation. $(k,g)$-planar cage vs planar $(k,g)$-cage. N: the cage is the one that achieves the lower bound. Thus, in this context it makes sense.} are the five platonic solids: the tetrahedron,
the cube, the octahedron, the dodecahedron and the icosahedron.
  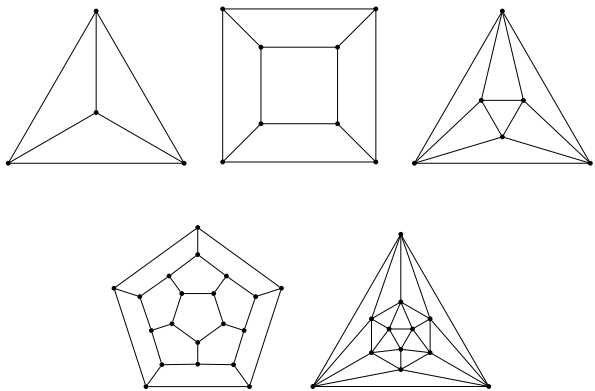
\begin{figure}[!ht]
    \centering
    \input{platonic.tex}
    \caption{The five platonic solids.}
    \label{fig:platonic}
  \end{figure}
\end{theorem}
\begin{proof}
  Let $G$ be a planar \cage{}. From our discussion above
  $k,g \in \{3,4,5\}.$ By the handshaking lemma $e=\frac{v k}{2}.$
  Using the Euler's characteristic equation\todo{Shall we add a
    reminder of this equation? No, it's almost folklore.}, we find that the number of faces of
  an embedding is such that $f=\frac{v(k-2)}{2} +2.$ Also note that
  the face lengths (i.e. the number of edges of each face) are bounded
  below by $g$ and the sum of the lengths of all faces of the
  embedding\todo{Are we assuming anything about an embedding or shall
    we just say ``any embedding''? N: I've written "an" embedding.} equals $2e,$ this is:
  $2e \geq g (\frac{v(k-2)}{2} +2).$

  Hence $v, k$ and $g$ have to satisfy the inequality
  \begin{equation}\label{eq:vbound1}
    v(2k-g(k-2))-4g \geq 0.
  \end{equation}
  Thus we must have
  \begin{equation}\label{eq:kg_pairs}
    2k-g(k-2)>0.
  \end{equation}
  It is easy to check that if $k,g \in \{3,4,5\},$ then only the pairs
  $(k,g)$ satisfying~\eqref{eq:kg_pairs} are
  $(3,3), (3,4), (3,5), (4,3)$ and $(5,3)$.

  Note that~\eqref{eq:vbound1} can be rewritten as
  \[
    v\geq\frac{4g}{2k-g(k-2)}.
  \]
  Also, note that each of the feasible pairs from above provides a lower
  bound for $v$. We argue that these lower bounds are in fact tight
  and the resulting $(k,g,v=n_p(k,g))$ triplet possibilities for $(k,g)$--cages are:
\begin{center}
(a)  $(3, 3, 4)$ \hspace{0.5cm}
(b)  $(3, 4, 8)$ \hspace{0.5cm}
(c)  $(3, 5, 20)$ \\
(d)  $(4, 3, 6)$ \hspace{0.5cm}
(e)  $(5, 3, 12)$ .
\end{center}

This follows as, $v(2k-g(k-2))-4g = 0$ if and only if the size of all
faces is precisely $g$. Thus, case (a) corresponds to a map on the
plane with vertex degree $3$ and face size $3,$ this an embedding of
the tetrahedron. Similarly, we can see that cases (b), (c), (d) and
(e) are the cube, the dodecahedron the octahedron and the icosahedron,
respectively\todo{This last assertion seems quite steep. Should we also
  prove uniqueness? N: No, the platonic solids are known to be unique once you've prescribed the size of the facial cycles}.
\end{proof}

\section{Biregular planar cages}

Now, we turn our attention to planar $(\{r,m\};g)$--cages \todo{Check for
  consistency: planar $*$-cage vs $*$-planar cage}.  We may assume
without loss of generality that $2\leq r<m$ and $3 \leq g,$ as we only consider simple graphs. Also, as we have
argued before, a planar graph must have a vertex of degree less than
or equal $5,$ thus, $2\leq r\leq 5.$

Here we start again by noticing that if $y$ is the number of vertices
of $G$ with degree $r$ and $x$ is the number of vertices of degree $m$
then, by the handshaking lemma, $2e=yr+xm.$ Also, as before, we have
$2e \geq g f $ and, $f=2-v+e.$ Combining these three equations
together we have:
\begin{equation}
y[r(2-g)+2g]+x[m(2-g)+2g]-4g\geq 0
\label{eqn:inequality}
\end{equation}

Using this equation, we may prove the following Lemma:

\begin{lemma}\label{lemma:biregularcages}
  An $(\{r,m\},g)$-graph can \todo{What do we mean by ``can be''? N: otherwise is for sure not planar, but we havent proved that the cage does actually exists.} be
  planar if and only if the triplet $(\{r,m\},g)$ is in one of the
  sets:
\begin{gather*}
\{(\{r, m\}, 3) | 2\leq r \leq 5, r <m \}  \hspace{0.5cm}
\{(\{r, m\}, 4) | 2\leq r \leq 3, r <m \}  \\
\{(\{r, m\}, 5) | 2\leq r \leq 3, r <m \}  \hspace{0.5cm} 
\{(\{2, m\}, g) | , 2 <m, 6\leq g \}
\end{gather*}
\end{lemma}
\begin{proof}
  Notice that, for equation~\eqref{eqn:inequality} to hold, we need at
  least one of the following: $r(2-g)+2g>0$ or $m(2-g)+2g>0.$ This is,
  an inequality of the form $y < \frac{2g}{g-2}=2+\frac{4}{g-2}$ must
  be satisfied for either $y=r$ or $y=m.$ As $r < m,$ if the equation
  is satisfied by $m,$ then it is automatically satisfied by $r.$
  Hence, $r$ satisfies the inequality.

  From these inequalities, it is easy to compute that the triplets
  $(\{r,m\},g) \in \mathbb{N}^3$ for which $r < 2+\frac{4}{g-2},$ with
  $2 \leq r$ and $2\leq g$ are as stated (See
  Figure~\ref{fig:inequality}).
\end{proof}

\begin{figure}
	\centering   
	\includegraphics[scale=0.2]{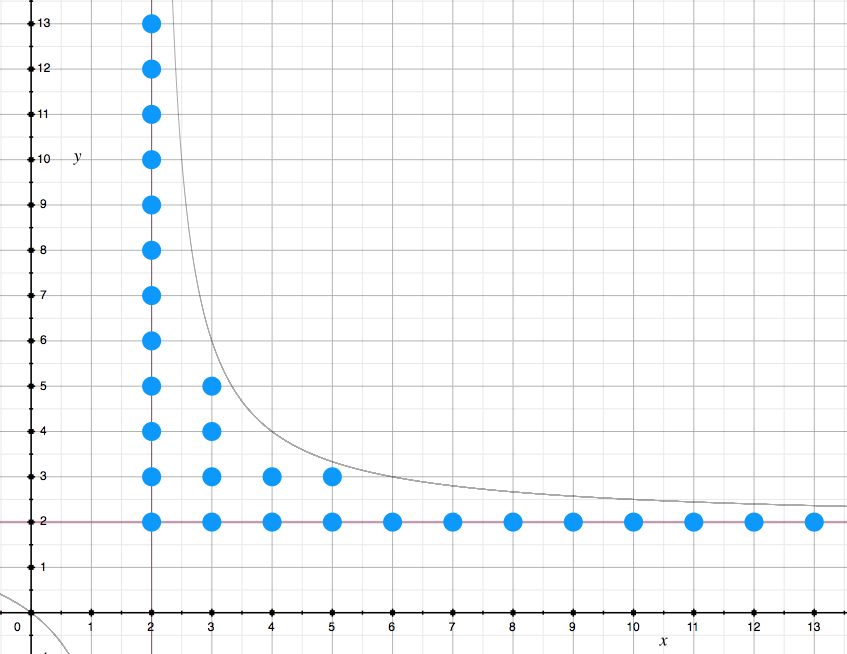} 
        \caption{The curve is the graph of the function
          $f(\alpha)=2+\frac{4}{\alpha-2}.$ The dots represent the integer
          pairs $(\alpha,\beta)$ for which $y<f(\alpha), \alpha \geq 2, \beta \geq 2.$ }
    \label{fig:inequality}
\end{figure}

\subsubsection{Lower bounds for planar biregular cages}
Recall that $n_p(\{r,m\};g)$ denotes the order of a $(\{r, m\};g)$-planar cage. In this section we will provide some easy lower bounds on $n_p(\{r,
m\};g)$.

\begin{lemma}\label{lemma:Badbound}
  For all $(\{r, m\};g)$ triplets in Lemma~\ref{lemma:biregularcages},
  $n_p(\{r, m\};g)$ increases as the number of vertices of degree $m$
  necessary to build a $(\{r, m\};g)$-planar cage increases.
\end{lemma}
\begin{proof}
  Recall Equation~\ref{eqn:inequality}:
  $y[r(2-g)+2g]+x[m(2-g)+2g]-4g\geq 0,$ where $x$ is the number of
  vertices of degree $m$ and $y$ is the number of vertices of degree
  $r.$ Here we may write $y=n_p-x,$ were $n_p= n_p(\{r, m\};g),$ to
  obtain:
  \begin{equation}\label{eqn:bound}
    n_p \geq \frac{4g+([r(2-g)+2g]-[m(2-g)+2g])x}{[r(2-g)+2g]}.
  \end{equation}

  As we have mentioned before, since $r < m$ then
  $[m(2-g)+2g] < [r(2-g)+2g]$. Also, from
  Lemma~\ref{lemma:biregularcages} we have
  $[r(2-g)+2g]>0$ for the triplets allowed. Given that $x>0,$ this implies that the right
  handside of Equation~\ref{eqn:bound} is positive, furthermore $n_p$
  increases as $x$ increases.
\end{proof}

As an easy corollary of Lemma~\ref{lemma:Badbound} we have the
following general lower bound:

\begin{corollary}\label{cor:badbound}
  For all $(\{r, m\};g)$ triplets in Lemma~\ref{lemma:biregularcages},
$$n_p(\{r, m\};g)\geq 1+ \frac{m(g-2)+2g}{r(2-g)+2g}.$$
\end{corollary}

\begin{corollary}\label{cor:lowerbounds}
  The following lower bounds for $n_p(\{r, m\};g)$ hold for the
  triplets $(\{r, m\};g)$ in Lemma~\ref{lemma:biregularcages}:
{\small
\begin{center}
\begin{tabular}{lllll}
 
 & $r$ & $m$ & $g$ & $n_p(\{r, m\};g)$ \\ 
\hline 
$(a)$ & $2$ & $m$ & $3$ & $\geq m+1$ \\ 

$(b)$ & $3$ & $m$ & $3$ & $\geq m+1$ \\ 

$(c)$ & $4$ & $m$ & $3$ & $\geq \max \{m+1, \frac{m}{2}+4 \}$ \\ 

$(d)$ & $5$ & $m$ & $3$ & $\geq m+7$ \\ 

$(e)$ & $2$ & $m$ & $4$ & $\geq m+2$ \\ 

$(f)$ & $3$ & $m$ & $4$ & $\geq m+5$ \\ 

$(g)$ & $2$ & $m$ & $5$ & $\geq 2m+1$ \\ 

$(h)$ & $3$ & $m$ & $5$ & $\geq 3m+11$ \\ 

$(i)$ & $2$ & $m$ & $6 \leq g \text{ even }$ & $\geq \frac{m(g-2)+4}{2}$ \\ 

$(j)$ & $2$ & $m$ & $6 \leq g \text{ odd }$ & $\geq \frac{m(g-1)+2}{2}$ \\ 
\end{tabular} 
\end{center}}
\end{corollary}
\begin{proof}
  For each case, will enlist the lemma or theorem that implies the better bound.\\
(a) Theorem~\ref{thm:Dgcages}.
(b) Theorem~\ref{thm:Dgcages}.
(c) Theorem~\ref{thm:Dgcages} and Corollary~\ref{cor:badbound}.
(d) Corollary~\ref{cor:badbound}.
(e) Theorem~\ref{thm:Dgcages}.
(f) Corollary~\ref{cor:badbound}.
(g) Theorem~\ref{thm:Dgcages}.
(h) Corollary~\ref{cor:badbound}.
(i) Theorem~\ref{thm:Ch_2_m_cages}.
(j) Theorem~\ref{thm:Ch_2_m_cages}.
\end{proof}

\subsection{Some properties of planar graphs}\label{sec:planargraphs}

In this section we will state some technical properties of planar
graphs, which will come in handy when presenting the biregular planar
cages in later sections. We include the proofs to all
properties in Section~\ref{sec:technical_lemmas}.

\begin{lemma}\label{lem:vertexdegree}
  \todo{Maybe we are missing some assumptions(?), please see
    ``counterexamples'' below. N: neither is a counterexaple, the bipartite graphs all fall in category c. }{A planar graph of order $m+1$, at least one vertex of degree $m\geq 3$ satisfies} exactly one of the following:
\begin{enumerate}[a.]
\item It has four vertices of degree $m$ and $m=3$.
\item It has three vertices of degree $m$ and $m=4$.
\item It has at most two vertices of degree $m$.
\end{enumerate}
\end{lemma} 

An \textit{outerplanar graph} is a graph that has a planar drawing for
which all vertices belong to the outer face of the drawing.

\begin{lemma}\label{lem:chorizo1}
  Let $G$ be a $(\{r,m\};g)$-planar graph, then the subgraph of $G$
  induced by all the vertices in the faces incident to a vertex $x,$
  $link_G(x),$ is an outer planar graph consisting of a (not
  necessarily disjoint) union of cycles (with or without chords) and
  paths, with at least $deg(x)$ vertices.
\end{lemma}

Let
$link_G(x)= (\cup_{i=1}^{k}C_i) \cup (\cup_{j=1}^{k'}
  P_j ),$ where the $C_i$ are the maximal induced cycles of
$link_G(x)$ and $P_j$ are its maximally induced trees in the sense
that they cannot be extended further without acquiring edges of some
other $P_j$ or $C_i.$ Define the \emph{intersection graph } of
$link_G(x),$ $I_x;$ whose set of vertices is
$\{c_1, \ldots, c_k, p_1, \ldots p_{k'}\}$ and where we put an edge
$c_i c_j$ or $c_i p_j$ every time the cycles $C_i, C_j$ or $C_i, P_j$
intersect in a vertex. Here we note that we do not consider $P_i$ to
$P_j$ incidences as it would be redundant by the definition of the $P_j's.$

\begin{lemma}\label{lem:chorizo2}
The intersection graph, $I_x,$ where $deg(x)=m$  is a simple graph, furthermore it is a forest.
\end{lemma}

\begin{lemma}
\label{lemma:outerplanar}
If $G$ is an outerplanar graph of order $\geq 4,$ such that all of its
vertices have degree at least $2$ then it has at least two
non-adjacent vertices of degree exactly $2.$
\end{lemma}

\subsection{Girth 3}

The next theorem states bounds for the \bicagesT[r,m] and mentions the record graphs. The proofs and descriptions of each graph family are distributed in the following subsections.

\begin{theorem} \label{thm:girth3} The \bicagesT[r,m] are as follows:
{\footnotesize
\begin{center}
\begin{tabular}{lllll}
$r$  & $m$ & $n_p(\{r,m\}, 3)$ & Graphs & Full list \\ 
 \hline 
$2$ & $m$ & $=m+1$ & $T'_{m-1}, T_{\frac{m}{2}}$ & yes \\ 

$3$ & $m$ & $=m+1$ & $W_m, M_m, \mathbf{W}_3 \text{ for } m=4$ & yes \\ 

$4$ & $m$ & $=m+2$ & $\mathbf{W}_{m+2}$ & - \\ 

$5$ & $m =6, 7$ & $= 2m+2$& $I_m$ & - \\ 

$5$ & $8\leq m \leq 13$ & $m+7 \leq n_p(\{r,m\}, 3) \leq 2m+2$& $I_m$ & - \\ 
$5$ & $14 \leq m$ & $m+7 \leq n_p(\{r,m\}, 3) \leq m+15$& $A_m$ & - \\ 
\end{tabular} 
\end{center}}
\end{theorem}

\subsubsection{Planar \bicagesT[2,m]}

First note that \bicageT[2,m] has at least $m+1$ vertices. Let $m=2l,$
and let $T_l$ be the graph consisting of $l$ triangles sharing a
common vertex, we will refer to it as the \textit{ windmill graph.}
(See Figure~\ref{fig:2mcages}.) Then $V(T_l)=2l+1=m+1,$ and $T_l$ has
exactly one vertex of degree $m,$ $m$ vertices of degree $2$ and by
construction it has girth $3.$

Let $l$ be even or odd, we denote by $T'_{l}$ the graph consisting of
$l-1$ triangles sharing a common edge, which we will refer to as the
\textit{pinwheel graph}. (See Figure~\ref{fig:2mcages}.) Then
$V(T'_l)=l+1,$ and $T'_l$ has exactly two vertices of degree $l,$
$l-1$ vertices of degree $2$ and by construction it has girth $3.$

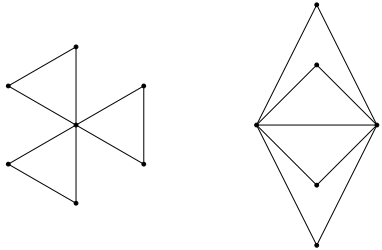
\begin{figure}[!ht]
  \centering
  \input{triangles.tex}
  \caption{The graphs $T_{3}$ and $T'_{4}$ respectively.}
  \label{fig:2mcages}
\end{figure}

Both $T'_{l}$ and $T_l$ are clearly planar graphs, thus, they are
clearly \bicagesT[2,m]. Furthermore they are the only planar cages.

\begin{lemma}
  Let $m\geq 3.$ The graphs $T'_{m-1}$ and $T_{m/2}$ for $m$ even are
  the only possible planar \bicagesT[2,m].
\end{lemma}
\begin{proof}
  The graphs are clearly planar \bicagesT[2,m]. To prove that they are
  unique we can argue, by Lemma~\ref{lem:vertexdegree}, that any
  planar \bicageT[2,m] has at most two vertices of degree $m.$

  Assume the graph $G, $ a planar \bicageT[2,m], has exactly one
  vertex of degree $m.$ Then $G-\{v\}$ is a $1$-regular graph with $m$
  vertices. Hence $G-\{v\}$ is a perfect matching, and necessarily $m$
  is even. Clearly, $G$ must be isomorphic to $T_{m/2}.$

  Now, assume the planar \bicageT[2,m], $G,$ has two vertices of
  degree $m$, say $u, v$ then $G-\{u, v\}$ is a $0$-regular graph with
  $m-1$ vertices. That is, an independent set of size $m-1, $ where
  each vertex is adjacent to both $u$ and $v$ in $G$ and $uv$ is also
  an edge of $G.$ Clearly, $G$ must be isomorphic to $T'_{m-1}$
\end{proof}

\subsubsection{Planar \bicagesT[3,m]}

We denote by $W_{l}$ the \textit{wheel} graph on $l+1$ vertices, see
Figure~\ref{fig:wheels}. This graph has one vertex of degree $l$ and
$l$ vertices of degree $3.$
 
 \begin{figure}[!ht]
    \centering
    \input{wheels.tex}
    \caption{From left to right, wheels {$W_5$, $W_6$ and $W_7$} respectively.}
    \label{fig:wheels}
  \end{figure}
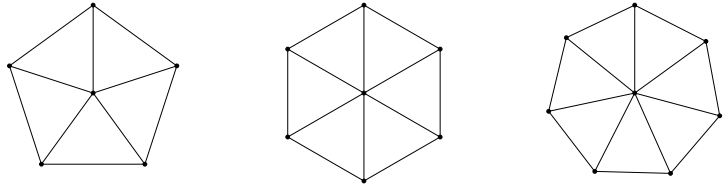
  
  Let $\mathbf{W}_{l}$ be the a graph obtained from the wheel graph
  $W_{l}$ by duplicating its vertex of degree $l$; {we will refer to
  $\mathbf{W}_{l}$ as a \textit{biwheel}. Notice that, $\mathbf{W}_{3}$ is isomorphic to $K_5$ minus one edge.

  Let $M_{l},$ the \textit{double windmill}, be the graph obtained
  from the windmill graph $T_{l-1}$ by duplicating its vertex of degree $l-1$ and making these two vertices adjacent.
  
\begin{lemma}
Let $m\geq 4,$ then, the only possible planar \bicagesT[3,m] are the double wheel $\mathbf{W}_{3}$ when $m=4$, the double windmill $M_{m}$ for $m$ odd and; the wheels $W_{m}$, for all $m$.
\end{lemma}
\begin{proof}
  Let $G$ be a planar \bicageT[3,m] with $m+1$ vertices. By
  Lemma~\ref{lem:vertexdegree} we may have the following cases:

  \emph{Case 1.} $G$ has three vertices of degree $4,$ $m=4$ and it
  has two vertices of degree $3.$ Hence $G$ is the biwheel,
  $\mathbf{W}_{3}$.

  \emph{Case 2.} $G$ has two vertices of of degree $m,$ say $\{u,v\},$
  and $m-1$ vertices of degree $3.$ Then $G-\{u,v \}$ is a $1$-regular
  graph. This is only possible if $G-\{u,v \}$ is a perfect matching,
  hence $m-1$ is even and $G$ is the double windmill $M_{m}$.

  \emph{Case 3.} $G$ has one vertex of degree $m,$ say $v,$ and $m$
  vertices of degree $3.$ Then $G-\{v\}$ is a $2$-regular graph with
  $m$ vertices, this is an $m$-cycle. Then $G$ is clearly the wheel
  $W_{m}.$
\end{proof}

\subsubsection{Planar \bicagesT[4,m]}

\begin{lemma}
The graph $\mathbf{W}_{m+2}$ is a planar \bicageT[4,m], with $m\geq 5$.
\end{lemma}
\begin{proof}
  It is clear that $\mathbf{W}_{m+2}$ is a planar \bigraphT[4,m] on
  $m+2$ vertices. Thus it remains to argue that $m+2$ is indeed the
  minimum number of vertices that a planar \bigraphT[4,m] can have.
  
  Suppose for a contradiction that there exists a planar \bicageT[4,m]
  $G$ that has fewer than $m+2$ vertices. Thus, $G$ has exactly $m+1$
  vertices. By Lemma~\ref{lem:vertexdegree}, since $m\geq 5,$ the
  graph $G$ must either have two vertices of degree $m$ and $m-1$
  vertices of degree $4$, or one vertex of degree $m$ and $m$ vertices
  of degree $4$. We will deal with each case separately.

  \emph{Case 1.} Suppose that $G$ has two vertices of degree $m$, say
  $u$ and $v$, and $m-1$ vertices of degree $4$, say
  $x_{1},\ldots,x_{m-1}$. Note that in this case $uv\in E(G)$ and
  $ux_{i},vx_{i}\in E(G)$ for $1\leq i\leq m-1$. Since $G-\{u,v\}$ is
  a $2$-regular graph then it must contain a cycle $C$, which is also
  a cycle in $G$. We succesively contract edges of $C$ in $G$ until we
  obtain a triangle, say $a,b,c$ in the graph $G'$. But now note that
  $G'[u,v,a,b,c]$ is isomorphic to $K_{5}$, thus $G$ contains $K_{5}$
  as a minor, which contradicts the planarity of $G$.

  \emph{Case 2.} Now assume that $G$ has exactly one vertex $u$ of
  degree $m$. Observe that $H=G-u$ is a planar $3$-regular
  graph. Furthermore, $H$ is outerplanar since in any planar drawing
  of $G$ every vertex is visible from $u$, thus all vertices of $H$
  are incident to the region of the plane that contained vertex
  $u$. Thus $H$ is a $3$-regular outerplanar graph and this
  contradicts Lemma~\ref{lemma:outerplanar}.
\end{proof}

\subsubsection{Planar \bicagesT[5,m]}\label{subsec:5_3_cages}
We start by describing the family $I_m$ of planar $(\{5,m\};3)$--graphs. In
Figure~\ref{fig:bigraphs53} we show schematics for $I_m$ for
$m=6,7,8.$

\begin{figure}[t]
  \centering
  \input{pinata.tex}
  \caption{The family of planar \bigraphsT[5,m], $I_m$. For each case, all
    vertices in the boundary are adjacent to an external vertex, to
    complete the graph.}
  \label{fig:bigraphs53}
\end{figure}
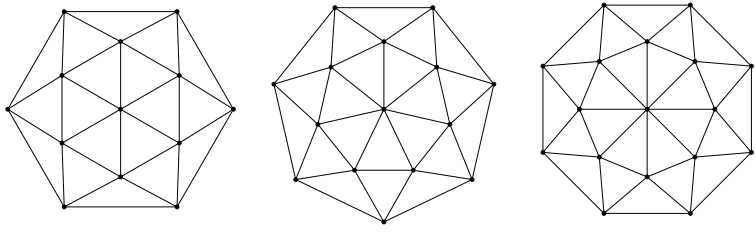

Let $I_m$ be the graph on $2m+2$ vertices,
$V(I_m)=\{x, x_0, \ldots, x_{m-1}, x', x'_0, \ldots, x'_{m-1}\}$ with
set of edges
\begin{gather*}
E(I_m)=\{x,x_i | 0\leq i \leq m-1\}  \cup \{x_i,x_{{i+1}\mod m} | 0\leq i \leq m-1\} \\
\cup \{x',x'_i | 0\leq i \leq m-1\} \cup \{x'_i,x'_{{i+1}\mod m} | 0\leq i \leq m-1\} \\
 \cup \{x_i, x'_i | 0\leq i \leq m-1\} \cup \{x_i,x'_{{i+1}\mod m} | 0\leq i \leq m-1\}.
\end{gather*} 
This construction proves that:
 
 \begin{proposition}\label{pinataprop} For all $6\leq m$, $\npbicageT[5,m]\leq 2m+2.$
\end{proposition}

An extensive computer search proves the non-existence of planar $(\{5,6\};3)$--graphs with $13$ vertices and planar $(\{5,7\};3)$--graphs with $14$ vertices, showing that in both of these cases $n_p(\{5,m\};3)=2m+2.$ These computations were performed using Magma \cite{MAGMA} and within it B. McKay's geng graph generator.

Additionally, consider the family of graphs depicted in
Figure~\ref{fig:betterr5g3}, $A'_m,$ each of this graphs has $10$
vertices at each end and $l=m-8$ vertices in between the ends. It
total it has $m+14$ vertices of which $m$ have degree $4$ and $14$
have degree $5.$ Let $A_m$ be the graph constructed by adding a vertex
to $A'_m$ which is connected to all vertices of degree $4.$ This
construction proves that:

\begin{proposition}\label{acordeonprop} For $m\geq 13,$ $n_p(\{5,m\},3)\leq m+15.$
\end{proposition}

In addition $2m+2\geq m+15$ when $m\geq 13.$ This completes the proof  of the bounds. Finally, we remark that the families given in Proposition \ref{pinataprop} and \ref{acordeonprop} give two non-isomorphic $(\{5,13\};3)$-planar cages of order $28$.

\begin{figure}[t]
  \centering
  \includegraphics[scale=0.2]{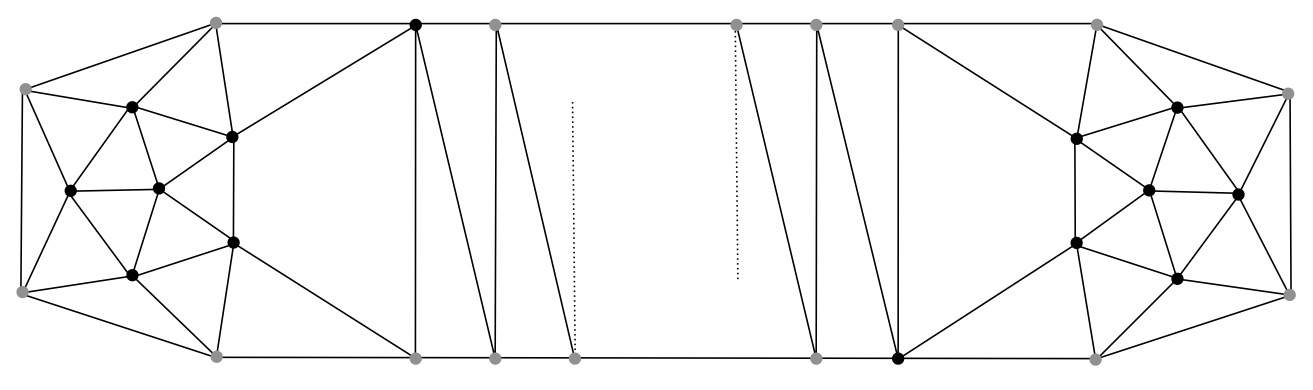} 
  \caption{A depiction of the graphs $A'_m$ with $16 \leq m$.}
  \label{fig:betterr5g3}
\end{figure}

\subsection{Girth $4$}

\begin{theorem} \label{thm:girth4}
The $(\{r,m\};4)$-cages are as follows:
{\scriptsize
\begin{center}
\begin{tabular}{lllll}
$r$  & $m$ & $n_p(\{r,m\}, 4)$ & Graphs & Full list \\ 
 \hline 
$2$ & $m$& $m+2$ & $K_{2,m}$  & yes \\ 

$3$ & $4 \leq m\leq 13$ & $ n_p(\{3,m\}, 4) = 2m+2$  & $D_m$  & - \\ 

$3$ & $14 \leq m$ & $m+\frac{4(m+1)}{5}+3 \leq n_p(\{3,m\}, 4) \leq m+4 \lceil \frac{m+1}{5}\rceil+3$  & -  & - \\ 

$3$ & $m=5k-1$ for $k\geq 3$ & $n_p(\{3,m\}, 4) =m+\frac{4(m+1)}{5}+3 $  & $Z_k$ & - \\ 
\end{tabular} 
\end{center}}
\end{theorem}

The statement of the theorem is a summary of the results presented in
the coming subsections. We may begin by remarking that
Lemma~\ref{lemma:biregularcages} implies that $2\leq r \leq 3, r <m$
and Corollary~\ref{cor:lowerbounds} points that
$m+2 \leq n_p(\{2,m\},4)$ and $m+5 \leq n_p(\{3,m\},4).$

\subsubsection{Planar \bicages[2,m]{4}}

\begin{lemma}
  The complete bipartite graph $K_{2,m}$ is the $(\{2,m\},4)$-planar
  cage.
\end{lemma}
\begin{proof}
  Corollary~\ref{cor:lowerbounds} proofs that
  $n_p(\{2,m\},4) \geq m+2.$ Therefore, we require at least $m+2$
  vertices and this bound is achieved by the complete bipartite graph
  $K_{2,m}$, which is clearly planar.

To prove uniqueness, it suffices to prove that if $G$ is a
$(\{2,m\},4)$-planar cage then it must have two vertices of degree
$m.$ Assume, to the contrary, that it only has one vertex of degree
$m,$ $v.$ Then $G \setminus \{v\}$ would be a graph with $m$ vertices
of degree $1$ and one vertex of degree $2,$ where two vertices
adjacent to $v$ in $G,$ cannot be adjacent, or there would be a
triangle. It's obvious that such a graph doesn't exist. Thus $G$ has
at least two vertices of degree $m$, $m$ vertices of degree $2,$
and it has no triangles, which characterizes $K_{2,m}.$
\end{proof}

\subsubsection{Planar \bicages[3,m]{4}}
In this section we will first show the lower bound for all cases and
then we will introduce families of graphs for which the lower
bounds are attained for some values of $m$.

\begin{lemma}\label{lem:linklowbound}
  Let $G$ be a $(\{3, m\}, 4)$-planar graph, $x$ be the unique vertex of degree $m$ in $G$, $k$ be the number of cycles in the decomposition of
  $link_G(x)= (\cup_{i=1}^{k}C_i ) \cup (\cup_{j=1}^{k'} P_j ),$ $c$
  the number of connected components of $link_G(x)$ and $ends(I_x)$ be
  the number of ends of $I_x,$ then $v(G) \geq 2m-k+c+ ends(I_x)+1.$

\end{lemma}
\begin{proof}

  Notice that in $link_G(x)$ two consecutive vertices can't both be
  neighbours of $x$, this implies that each of the cycles and trees in
  the decomposition
  $link_G(x)= (\cup_{i=1}^{k}C_i ) \cup (\cup_{j=1}^{k'} P_j),$ inherits this property. Hence,
  for each $C_i$ and $P_j$ we have that $v(C_i)\geq 2m_i$ and
  $v(P_i)\geq 2m_j+1$, respectively, where $m_i, m_j$ represent the
  number of vertices adjacent to $x$. Here
  $\sum_{i=1}^{k}m_i +\sum_{j=1}^{k'} m_j=m$

By Lemma~\ref{lem:chorizo2} we know that $I_x$ is a forest, this
implies that the number of vertices belonging to the pairwise
intersections $C_i \cap P_j$ or $C_i \cap C_j$ is exactly $k+k'-c.$ Thus,
we have:
$v(link_G(x))\geq \sum_{i=1}^{k} v(C_i) +\sum_{j=1}^{k'}
v(P_j)-(k+k'-c)=2m+k'-(k+k'-c)=2m-k+c.$

As $v(G)= v(link_G(x))+ v(G \setminus link_G(x)),$ we now look at how
the structure of $I_x$  helps bound
$v(G\setminus link_G(x)).$

\emph{Claim 1. The end vertices of $I_x$ are vertices that represent a
  cycle-type component.}

Else, there would be terminal vertices in $link_G(x)$, which would be
vertices of degree $2$ in $G.$

\emph{Claim 2. In $I_x$ edges only exist among pairs of vertices
  representing one tree and one cycle.}

If two components representing cycles in the decomposition of
$link_{\color{blue}G}(x)$ intersect, the intersecting vertex would have degree four, but the vertices in $link_G(x)$ have degree $3$ or $2$. Recall that, by Lemma \ref{lem:chorizo2},  cycles can't intersect in more than one vertex.

\begin{figure} [t]
\begin{center}
\includegraphics[scale=0.2]{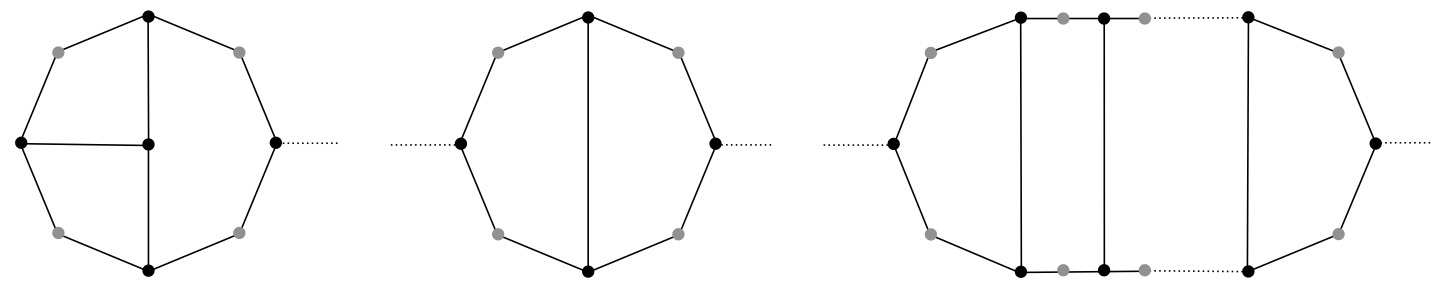} 
\caption{The graphs $F, E_4$ and a depiction of  $E_l$ for $l\geq 8.$ }
\label{fig:g4pieces}
\end{center}
\end{figure}

\emph{Claim 3. $v(G \setminus link_G(x))\geq ends(I_x)+1.$}

Notice that for each $C_i$ we have at least $m_i$ vertices of degree
$2$ which are not adjacent to $x.$ This implies that in the embedding
of the graph there must be some edges emanating from said vertices,
either reaching some other vertices in $C_i,$ vertices in the components of
$link_G(x)$ adjacent to $C_i$ or some vertices not in $C_i$ (lying in
the area enclosed by $C_i$ in the embedding).

Suppose there are some chords (non-crossing in the embedding) among
the vertices of $C_i$ non adjacent to $x.$ As, there are no triangles
in $G$, these vertices have to be at distance at least three in the
cycle. So, there would have to be at least two vertices in $C_i$ not
adjacent to $x$ and not adjacent to any other vertex in $C_i.$ Thus, if there
are no additional vertices of $G$ in the area enclosed by $C_i,$ then
$C_i$ has exactly $2m_i$ vertices, $m_i-2$ which are paired by non-crossing chords and two which have to be connected to other parts of $link_G(x).$ (See the
graphs in the center and right side in Figure~\ref{fig:g4pieces}.)

Thus, for $C_i$ to represent an end vertex of $I_x$ there is no remedy
but to have one additional vertex of $G$ in the area enclosed by
$C_i,$ connected to three vertices of $C_i$ non adjacent to $x$ and
exactly one vertex of $C_i$ non adjacent to $x$ which, connects to
another part of $link_G(x).$ (See the leftmost graph in
Figure~\ref{fig:g4pieces}.) The result follows from this claim.
\end{proof}

\begin{corollary}
  Let $G$ be a $(\{3, m\}, 4)$-planar graph then for
  $4 \leq m\leq 13,$ $2m + 2 \leq v(G)$ and for $14 \leq m,$
  $\frac{9m+19}{5} \leq v(G).$
\end{corollary}
\begin{proof}
As we have proved that $v(G)$ increases as the number of vertices of degree $m$ increases, then we will assume for this lower bound that $G$ has exactly one vertex of degree $m$, $x$.
  Let $G$ be a $(\{3, m\}, 4)$-planar graph and $x$ be a vertex of
  degree $m$ in $G$. From the equation $v(G) \geq 2m-k+c+ ends(I_x)+1$
  in the previous lemma, we can observe that, when $G$ is the smallest
  possible $(\{3,m\}, 4)$-planar graph, then $G$ is connected,
  has big $k$ and small $ends(I_x).$

As $ends(I_x)\geq 1$ we have that $v(G) \geq 2m-k+ ends(I_x)+ 2.$ Obviously this bound will decrease as $k$ increases.

For $I_x$ to have few ends the optimal case is when it is a path where
each $P_j$ represents a path of length two or one. Thus, we observe
that, for the extremally small $(\{3, m\}, 4)$-planar graphs,
$ends(I_x)=1 \text{ or } 2.$

\emph{Claim 1.  For $4 \leq m\leq 13,$ $2m + 2 \leq v(G)$. }

From the proof of the Lemma~\ref{lem:linklowbound} we can
observe that the smallest $C_i$ that can represent an end vertex of
$I_x$ has $9$ vertices and $m_i=4$ and the smallest $C_i$ that can
represent a non-end of $I_x$ has $8$ vertices and $m_i=4.$ (See
Figure~\ref{fig:g4pieces}). That is, if  $m$ is small enough we have no alternative but to have $k=1, ends(I_x)=1,$ and $v(G) \geq 2m+ 2,$ this also holds if $k=2, ends(I_x)=2,$ for example.

On the other hand, for  $2m-k+ ends(I_x)+ 2 < 2m+1$ to be satisfied we need $k\geq 3.$ As each $C_i$ is followed by $P_j$ then, for a given $m,$ the best possible situation is to have as many $P_j$ of length $2$ between
each $C_i$ of size $8.$ This implies that for $k\geq 3$ we must be
able to have at least three cycles of size $8$ and two paths of length
two in the decomposition of $link_G(x)$. Thus, the total degree of $x$
is $m\geq 4(3)+2=14.$ This implies the claim.

\emph{Claim 2. For $14 \leq m,$ $\frac{9m+19}{5} \leq v(G).$}

Here, by Claim 1, we may have $k\geq 3$ and $ends(I_x)=2.$ Thus, we
need to find the greatest $k$ such that $m \geq 4k+(k-1)= 5k-1$ or,
equivalently, the greatest $k \leq \frac{m+1}{5}$ and the result
follows by plugging in this values into the bound's equation.
\end{proof}



Let $D_m$ be the graph whose set of vertices is
$\{x_0, \ldots x_{2m-1}, y_0,y_1\}$ and whose set of edges is
$\{x_i x_{i+1} | i=0,\ldots, 2m-2\} \cup \{x_1, x_{2m-1}\} \cup \{y_0
x_{2j} | j=0,\ldots, m\} \cup \{y_1 x_{2j-1} | j=0,\ldots, m\}.$ See
Figure~\ref{fig:subdivided_cube}. This graph is clearly a
$(\{3, m\}, 4)$-planar graph with $2m+2$ vertices, implying;

\begin{proposition}
For $4\leq m \leq 13,$ $n_p(\{3, m\}, 4)= 2m+2$.
\end{proposition}

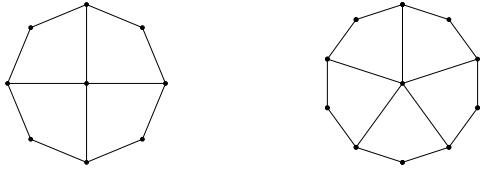
\begin{figure}[h]
  \centering
  \input{cubdividido.tex}
  \caption{$D_m$ graphs for $m=4,5$. In each case, all vertices of degree two in the boundary are adjacent to an external vertex.}
  \label{fig:subdivided_cube}
\end{figure}


Let $F$ be the graph whose set of vertices is
$\{x, x_0, \ldots, x_7\}$ and whose set of edges is
$\{ x_i x_{i+1} | i=0, \ldots, 6\} \cup \{x_0 x_{7} \} \cup \{x x_{2i}
| i=0,1,2\}.$ Let $E_4$ be the graph whose set of vertices is
$\{ x_0, \ldots, x_7\}$ and whose set of edges is
$\{ x_i x_{i+1} | i=0, \ldots, 6\} \cup \{x_0 x_{7} \} \cup \{x_0
x_{4}\}.$ A depiction of these graphs can be found in
Figure~\ref{fig:g4pieces}.

For $k\geq 3, $ let $Z'_k$ be the graph resulting from joining two
copies of $F,$ $k-2$ copies of $E_4$ and $k-1$ paths of length $2,$
$P$ as follows: start with a copy of $F,$ and join it to a copy of the
path $P$ by identifying its only even labeled vertex of degree two
with an end of $P$, join the remaining end of $P$ to a copy of $E_4$
by one of it's two even labeled vertices of degree two, now repeat
this procedure subsequently joining copies of $E_4$ to copies of $P$
and ending with the second copy of $F$. This graph has $9k+1$
vertices, out of which $5k-1$ have degree two and the remaining
vertices have degree three. A depiction of $Z'_k$ can be found in
Figure~\ref{fig:G4_cage}.

\begin{figure}[h]
\begin{center}
\includegraphics[scale=0.2]{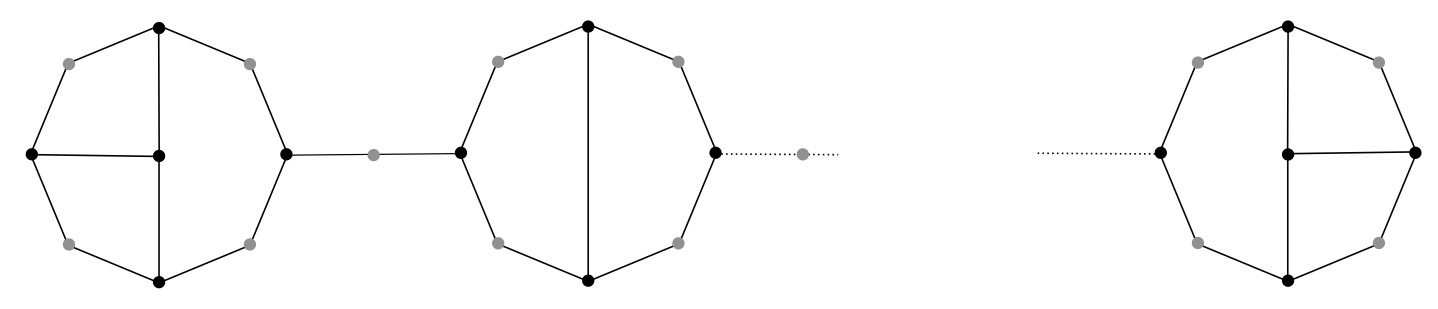} 
\caption{A graph $Z'_k$ with $k=4$. In order to build $Z_k$ we may join all degree two vertices in $Z'_k$ to the same vertex of degree $m=9k+1.$ }
\label{fig:G4_cage}
\end{center}
\end{figure}

Let $Z_k$ be the graph whose set of vertices is $V(Z'_k) \cup \{x^*\}$
and whose set of edges is
$E(Z'_k) \cup \{x^* y | y \in V(Z'_k) \text{ and } deg(y)=2\}.$
Observe that $Z_k$ has $9k+2$ vertices, out of which one has degree
$5k-1$ and the remaining vertices have degree 3. This construction
proves:

\begin{proposition}
$n_p(\{3, m\}, 4)\leq \frac{9m+19}{5}$, for $m=5k-1$ and $k\geq 3.$
\end{proposition}

Finally, note that in the previous construction we may delete up to two degree $2$ (diametrically opposite) vertices from each $E_4$ or $F,$ without violating the girth condition in $Z'_k.$ Thus, for $m > 14$ we may find the smallest $m^*$ such that $m<m^*$ and $m* =5k^*-1$ for some $k^*$, make the construction $Z'_{k^*}$ and then remove $m^*-m$ degree two vertices from such construction to obtain an improved upper bound for any $m.$  A simple computation proves $m^*=5 \lceil \frac{m+1}{5}\rceil-1$ and the number of vertices of such construction will be $m+4 \lceil \frac{m+1}{5}\rceil+3.$ Hence, we have:

\begin{proposition}
$n_p(\{3, m\}, 4)\leq m+4 \lceil \frac{m+1}{5}\rceil+3$, for $m>14.$
\end{proposition}

\subsection{Girth $5$}

\begin{theorem} \label{thm:girth5}
The $(\{r,m\};5)$-cages are as follows:
{\footnotesize
\begin{center}
\begin{tabular}{lllll}
$r$  & $m$ & $n_p(\{r,m\}, 5)$ & Graphs & Full list\\ 
 \hline 
$2$ & $m$& $2m+1$ & $O_{m,5}$ & yes \\ 
$2$ & $m \text{ even }$& $2m+1$ & $F_{m,5}$ & yes \\ 
$3$ & $4 \leq m \leq 5$& $3m+11 \leq n_p(\{3,m\},5) \leq 6m+2$ & $P_m$  & - \\ 
$3$ & $6 \leq m$, $m \text{ even }$& $3m+11 \leq n_p(\{3,m\},5) \leq 3m+2\lfloor\frac{m-6}{4}\rfloor+21$ & $B_m$ & - \\ 
$3$ & $6 \leq m$, $m \text{ odd }$& $3m+11 \leq n_p(\{3,m\},5) \leq 3m+2\lfloor\frac{m-5}{4}\rfloor+22$ & $B_m$ & - \\ 
\end{tabular} 
\end{center}}

\end{theorem}

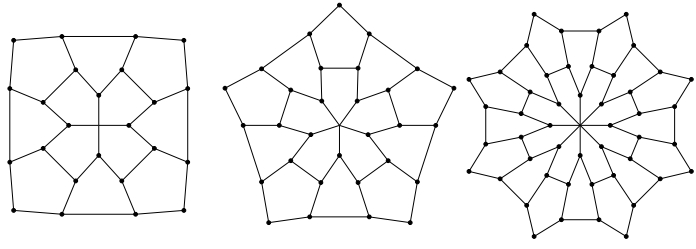
\begin{figure}[t]
  \centering
  \input{pentapinatas.tex}
  \caption{A family of planar \bigraphs[3,m]{5}. In all cases, the
    vertices of degree two in the boundary are adjacent to an external
    vertex.}
  \label{fig:pentapinatas}
\end{figure}

The proof of the theorem and description of each graph family is split
in the following subsections. However, we may remark that
Lemma~\ref{lemma:biregularcages} implies that $2\leq r \leq 3, r <m$
and Corollary~\ref{cor:lowerbounds} points that
$n_p(\{2,m\},4) \geq 2m+1$ and $n_p(\{3,m\},4) \geq 3m+11.$

\subsubsection{Planar \bicages[2,m]{5}}

This case will follow as a consequence of a more general construction in the next section. We will define the graphs $O_{m,g}$ and $F_{m,g}$  for $g\geq3$, and prove that they are the only planar cages in Lemma~\ref{lem:2mcages}. 

\subsubsection{Planar \bicages[3,m]{5}}

Note that, if we allow two vertices of degree $m$ in
\bigraphs[3,m]{5}, from Lemma~\ref{lemma:Badbound} we obtain that such
graphs must have at least $6m+2$ vertices.  Here we present an
infinite family of \bigraphs[3,m]{5}, $P_m,$ meeting that bound, see
Figure~\ref{fig:pentapinatas}. This proves that indeed:

\begin{proposition} For $4\leq m$, $n_p(\{3,m\},5)\leq 6m+2.$
\end{proposition}

Furthermore, consider the family of graphs graph $B'_m$ depicted in
Figure~\ref{fig:betterG5}. This family has
$v=3m+2\lfloor\frac{m-6}{4}\rfloor+20$ of which $m=2l+6$ vertices  {and  $l\geq 0$,}  have
degree two and the rest have degree $3.$ Let $B_m$ be the graph in
which we add a vertex to $B'_m$ joined to all its vertices of degree
two. This construction proves that:

\begin{proposition} For $6\leq m,$ even, $n_p(\{3,m\},5)\leq 3m+2\lfloor\frac{m-6}{4}\rfloor+21.$
\end{proposition}

Finally, note that in the construction of $B_m$ there's several degree two vertices that can be deleted without decreasing the girth. Thus for even $m\geq 7$ we may construct $B_{m+1}$ and then delete one vertex to obtain a graph with $3m+2\lfloor\frac{m-5}{4}\rfloor+22$ vertices.

\begin{proposition} For $7\leq m,$ odd, $n_p(\{3,m\},5)\leq 3m+2\lfloor\frac{m-5}{4}\rfloor+22.$
\end{proposition}

\begin{figure}[t]
  \centering
  \includegraphics[scale=0.18]{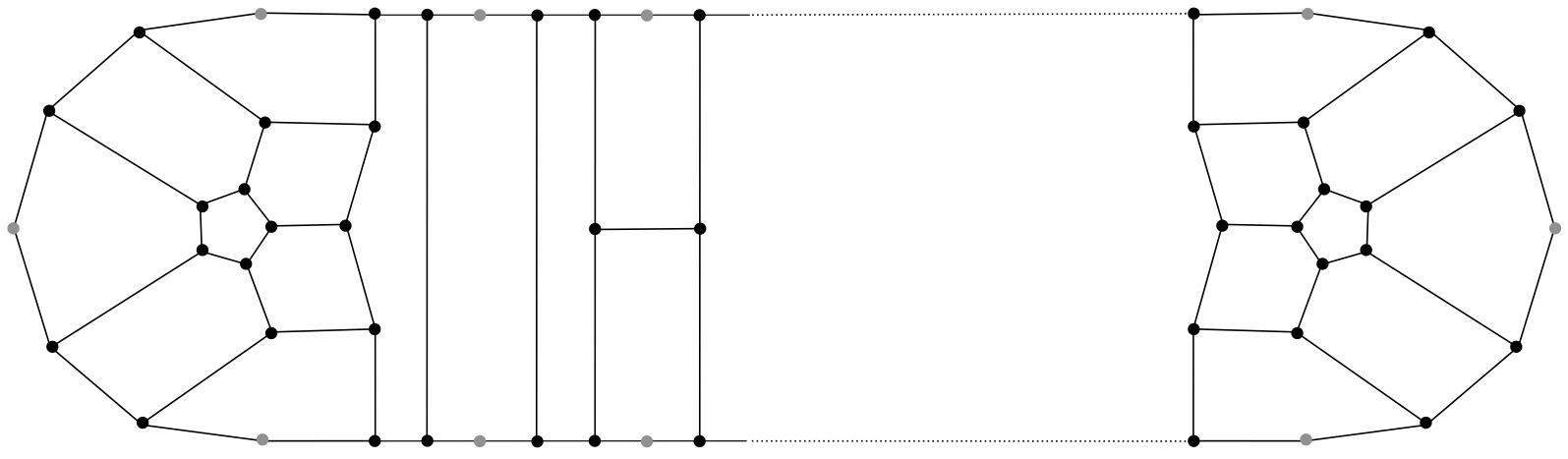}
  \caption{A family of planar \bigraphs[3,m]{5}. In all cases, the
    vertices of degree two in the boundary are adjacent to an external
    vertex.}
  \label{fig:betterG5}
\end{figure}

\subsection{Girth $g\geq 6$}

\begin{theorem} \label{thm:girth6}
The $(\{r,m\};g)$-planar cages with $g \geq 6$ are as follows:

\begin{center}
\begin{tabular}{lllll}
$r$  &$g$ & $n_p(\{r,m\}, 6)$ & Graphs & Full list \\ 
 \hline 
$2$ &  even & $= \frac{m(g-2)+4}{2}$ &  $O_{m,g}$ & yes \\ 
$2$ &  odd & $= \frac{m(g-1)+2}{2}$ &  $O_{m,g}, F_{m,g} \text{ for even } m$ & yes \\ 
\end{tabular} 
\end{center}

\end{theorem}

By Lemma~\ref{lemma:biregularcages}, we know that the only possible
cases are when $r=2$ and $m\geq 3.$

For any $g$ and $m\geq 3,$ let $O_{m, g}$ be the graph on
$(m-1)\lceil \frac{g-2}{2}\rceil+\lfloor \frac{g-2}{2}\rfloor +2$
vertices that consists of $m$ independent paths between two vertices
$u,v$, $(m-1)$ being of length $\lceil \frac{g}{2}\rceil$ and one of
length $\lfloor \frac{g}{2}\rfloor.$ It is easy to see that this is a
$(\{2,m\}; g)$-graph which is embeddable in the plane.

For any $g$ and $m$ even let $F_{m,g}$ be the graph formed by
$\frac{m}{2}$ cycles of lenght $g$ all incident in one vertex. See
Figure (reference). This graph has $\frac{m}{2}(g-1)+1$ vertices and
is clearly planar.

\begin{lemma}\label{lem:2mcages}
  The only possible $(\{2,m\};g)$-planar cages are $O_{m, g}$ for any
  $2<m$ and $g\geq 6,$ and $F_{m,g}$ for any $2<m$ even and $g\geq 5$
  odd.
\end{lemma}
\begin{proof}
  It is not difficult to check that the numbers of vertices of
  $O_{m,g}$ and $F_{m,g}$ reach the lower bounds in
  Corollary~\ref{cor:lowerbounds}, for the cases mentioned.

  Now, for the characterization part, we will first argue that the
  number of vertices $x$ of degree $m$ is at most two. We will proceed
  by contradiction. By Corollary~\ref{cor:lowerbounds} if $x\geq 3$
  then the number of vertices necessary to have a planar
  $(\{2,m\};g)$-graph, $v$ is such that:
  $$v \geq \frac{3(g-2)m-2g}{4}+3.$$
  It is easy to check that this number is always higher than the
  number of vertices of $O_{m,g}$ and $F_{m,g}.$ Thus, a
  $(\{2,m\};g)$-planar cage has at most $2$ vertices of degree $m.$

  Let $x=1,$ $G$ be a planar cage and $v$ be the vertex in $G$ of
  degree $m$. Then $G \setminus \{v\}$ is a graph with $m$ vertices of
  degree $1$ and $V(G)-m-1$ vertices of degree two. Hence,
  $G \setminus \{v\}$ has to be exactly the union of $\frac{m}{2}$
  disjoint paths. This clearly implies that $G$ is $F_{m,g}.$

  Let $x=2,$ $G$ be a planar cage and $v, u$ be the vertices in $G$ of
  degree $m$. Then $G \setminus \{v, u\}$ is a graph with $2m$
  vertices of degree $1$ and $V(G)-2m-1$ vertices of degree
  two. Hence, $G \setminus \{v, u\}$ has to be exactly the union of
  $m$ disjoint paths. This clearly implies that $G$ is $O_{m,g}.$
\end{proof}

\section{Proofs of Lemmas in Section~\ref{sec:planargraphs}}
\label{sec:technical_lemmas}
\subsection{Proof of Lemma~\ref{lem:vertexdegree}}

\begin{lemma*}
  A planar graph with at least one vertex of degree $m\geq 3$ that has
  exactly $m+1$ vertices satisfies one of the following:
\begin{enumerate}[a.]
\item It has four vertices of degree $m$ and $m=3$
\item It has three vertices of degree $m$ and $m=4$
\item It has at most two vertices of degree $m$
\end{enumerate}
\end{lemma*} 
\begin{proof}
Assume $G$ is such a graph.

\begin{enumerate}[a.]
\item Assume that $G$ has four vertices, say $\{v_1, v_2, v_3, v_4\}$
  of degree $m.$ The graph induced by this four vertices is $K_4.$ If
  $m=3,$ $G$ is the tetrahedron. If $m\geq4$ then every vertex in
  $v \in V(G-\{v_1, v_2, v_3, v_4\})$ is connected to all of
  $\{v_1, v_2, v_3, v_4\}.$ Hence $\{v, v_1, v_2, v_3, v_4\},$ induces
  a $K_5,$ making $G$ non planar. Thus this is only possible when
  $m=3.$

\item Assume that $G$ has exactly three vertices of degree $m,$ say
  $\{v_1, v_2, v_3\}.$ Then the graph induced by these three vertices
  is $K_3.$ Also, every vertex in $v \in V(G-\{v_1, v_2, v_3\})$ is
  connected to all of $\{v_1, v_2, v_3\}.$ If
  $|V(G-\{v_1, v_2, v_3\})|=1,$ then $m=3$ and we would have the
  previous case, where all of the four vertices have degree $m$. Hence
  $m\geq 4.$ We will argue that $m=4.$

  First assume that $m\geq 5$ then $|V(G-\{v_1, v_2, v_3\})|\geq 3,$
  then for all triplets $\{u_1, u_2, u_3\} \in V(G-\{v_1, v_2, v_3\})$
  we have that the graph induced by
  $\{u_1, u_2, u_3\} \cup \{v_1, v_2, v_3\}$ contains $K_{6,6},$ hence
  $G$ would not be planar.

  Finally, if $m=4$ then $|V(G-\{v_1, v_2, v_3\})|= 2,$ and $G$ would
  be the $1$-skeleton of a double pyramid with triangular
  base. Clearly, a planar graph.

\item Clearly, the graph must have at most two vertices of degree $m$.
\end{enumerate}
\end{proof}

\subsection{Proof of Lemma~\ref{lem:chorizo1}}

\begin{lemma*}
  Let $G$ be a $(\{r,m\};g)$-planar graph, then the subgraph of $G$
  induced by all the vertices in the faces incident to a vertex $x,$
  $link_G(x),$ is an outer planar graph consisting of a (not
  necessarily disjoint) union of cycles (with or without chords) and
  paths, with at least $deg(x)$ vertices.
\end{lemma*}
\begin{proof}
  We may assume without loss of generality that $G \setminus x$ has
  only one connected component.

  Let $X=\{x_1, \ldots x_l\}$ be an ordered set that labels the
  vertices in $link_G(x)$, in the order they appear around $x$ in the
  embedding. Here, by assumption, we need to have $deg(x) \leq l $
  where there may be some repetitions of vertices.

  Assume that there are indeed some repetitions in $X,$ say
  $x_i=x_{j}$ and $x_{i'}=x_{j'}$ then we can never have
  $i < i' <j < j'$, where all subíndices are taken $\mod l$, otherwise
  we would violate the planarity of $G$. This proves the lemma.
\end{proof}

\subsection{Proof of Lemma~\ref{lem:chorizo2}}

\begin{lemma*}
  The intersection graph, $I_x,$ with $deg(x)=m$ is a simple graph,
  furthermore it is a forest. 
\end{lemma*}
\begin{proof}
  We may assume without loss of generality that $link_G(x)$ is
  connected, as the result easily generalizes from the connected case
  to the disconnected case. Hence, we need to prove that $I_x$ is a
  simple graph which is a tree.

\begin{enumerate}[a.] 
\item We will begin by arguing that there are no multiple edges.

  Suppose $c_i c_j$ is a multiple edge, then $C_i, C_j$ intersect in
  more than one vertex, say they are $u$ and $v.$

\begin{itemize}
\item If this vertices are adjacent in both cycles, then we could have
  considered the union of $C_i, C_j$ as a single cycle with a chord.
\item Hence, we may assume that $u$ and $v$ are non adjacent in at
  least one of the two cycles. This is, there are at least three paths
  of length at least two between $u$ and $v.$ These paths divide the
  plane in at least three regions, thus the union of $C_i, C_j$
  wouldn't be outerplanar, contradicting the outerplanarity of
  $link_G(x).$
\end{itemize} 

The case where we suppose $c_i p_j$ is a multiple edge is proved
similarly.

\item We will now argue that there are no cycles in $I_x.$ 

  Suppose, to the contrary, that there is a cycle $\mathcal{C}$ in
  $I_x,$ then there is a cycle of $G$ contained in the union of the
  cycles and paths corresponding to each of the vertices in
  $\mathcal{C}$. This cycle divides the plane in two connected
  components, say $\mathcal{C}^+$ and $\mathcal{C}^-.$ We may assume
  without loss of generality that $x \in \mathcal{C}^+,$ by
  construction there can be no vertices of $link_G(x)$ in
  $\mathcal{C}^-$ and all vertices of $link_G(x)$ have to be visible
  from $x$ in $\mathcal{C}^+.$ Hence there is a cycle induced by the
  cycles and paths corresponding to the vertices of $\mathcal{C}$
  which necessarily contains all the vertices of the cycles
  corresponding to vertices in $\mathcal{C}$, otherwise outerplanarity
  would be violated, but this contradicts the maximality of the cycles
  represented in $I_x.$
\end{enumerate}
\end{proof}

\subsection{Proof of Lemma~\ref{lemma:outerplanar}}
\begin{lemma*}
  If $G$ is an outerplanar graph of order $\geq 4,$ such that all of
  its vertices have degree at least $2$ then it has at least two
  non-consecutive vertices of degree exactly $2.$
\end{lemma*}
\begin{proof}
  We will denote the number of vertices of $G$ as $n.$ We will proceed
  by induction on $n$. Note that the graph may have different
  connected components and that the result holds trivially for cycles.

\begin{enumerate}[a.]
\item \emph{$n=4$}. If $G$ is a cycle the result follows. Assume that
  $G$ is not a cycle and let $(v_1, v_2, v_3, v_4)$ be the cycle in
  $G$ that bounds the outerface of the drawing. Then, the only
  additional edge of $G$ not in $(v_1, v_2, v_3, v_4)$ is either
  $v_1v_3$ or $v_2v_4.$ In the first instance the two non consecutive
  vertices of degree two are $v_2,v_4$, and in the second instance
  they are $v_1,v_3.$
\item \emph{$n\leq k$.} Assume that any outer planar graph of order at
  most $k$ such that all of its vertices have degree at least $2$,
  contains at least two non consecutive vertices of degree $2.$
\item \emph{$n=k+1.$} If $G$ is a cycle the result follows trivially. Also, if $G$ has at least two connected components the result follows.

  Thus, we may assume that $G$ has a unique connected component. Let
  $(v_1, v_2, \ldots, v_n)$ be the cycle that bounds the outerface of
  the drawing. As $G$ is not a cycle, there is an edge $v_iv_j$ that
  splits the graph in to two smaller outerplanar subgraphs $G_1$ and
  $G_2,$ which both contain a copy of the edge $v_iv_j.$

  If both $G_1$ and $G_2$ are of order at least $4,$ then each graph
  contains a vertex of degree $2,$ different from $v_i$ and $v_j, $
  and the result follows.

  Thus we only have to prove it for when either $G_1$ or $G_2$ are of
  order $3.$ We can assume without loss of generality that $G_1$ is of
  order $3.$ Here the vertex in $G_1$ different from $v_i$ and $v_j, $
  has degree two. As for $G_2,$ if its order is $\geq 4$ then the
  result follows, by the induction hypothesis. Otherwise, the vertex
  in $G_2$ different from $v_i$ and $v_j, $ has degree two.
\end{enumerate}
\end{proof} 

\section{Conclusions}

For the $(\{3,m\};4)$-- graphs where we have not reached the lower bound, we believe it to be unlikely that other constructions can improve the bounds provided. For all the other cases where there is still room for improvement it would be nice to see such improvements, either in the form of improved lower bounds or constructions.

Finally we consider that studying the \emph{biregular planar cage problem} for other surfaces, oriented or non-oriented, will lead to nice discoveries.

\emph{Acknowledgements.} The authors would like to thank D. Leemans for his help performing the Magma computations in Subsection \ref{subsec:5_3_cages}.

The research was partially supported by projects PAPIIT-M\'exico  IN107218 and IN106318, CONACyT-M\'exico 282280 and UNAM-CIC `Construcci\'on de Jaulas Mixtas'.


\end{document}

%% file: platonic.tex
\usetikzlibrary{decorations,decorations.markings,decorations.text,calc,arrows}

\begin{tikzpicture} [scale=0.9,
  _vertex/.style ={circle,draw=black, fill=black,inner sep=0.5pt},
  r_vertex/.style={circle,draw=red,  fill=red,inner sep=0.5pt},
  g_vertex/.style={circle,draw=green,fill=green,inner sep=0.5pt},
  b_vertex/.style={circle,draw=blue, fill=blue,inner sep=0.5pt},
  _edge/.style={black,line width=0.2pt},
  r_edge/.style={red,line width=0.3pt},
  g_edge/.style={green,line width=0.3pt},
  b_edge/.style={blue,line width=0.3pt},
 every edge/.style={draw=black,line width=0.3pt},
t_edge/.style={cap=round, ultra thick}]

\def\xmove{5.5}
\def\ymove{4}

\begin{scope}[xshift=-3cm,yshift=3cm]
\def\rad{1.5}

\coordinate (r) at (90:\rad);
\coordinate (g) at (330:\rad);
\coordinate (b) at (210:\rad);

\coordinate (v1) at (barycentric cs:r=19,g=0,b=0);
\coordinate (v2) at (barycentric cs:g=19,b=0,r=0);
\coordinate (v3) at (barycentric cs:b=19,r=0,g=0);
\coordinate (v4) at (barycentric cs:r=1,g=1,b=1);

\draw[_edge] (v1)--(v2);
\draw[_edge] (v2)--(v3);
\draw[_edge] (v3)--(v1);
\draw[_edge] (v1)--(v4);
\draw[_edge] (v2)--(v4);
\draw[_edge] (v3)--(v4);

\node[_vertex] (V1) at (v1) {};
\node[_vertex] (V2) at (v2) {};
\node[_vertex] (V3) at (v3) {};
\node[_vertex] (V4) at (v4) {};
\end{scope}

\begin{scope}[xshift=0cm,yshift=3.4cm]
\def\Rad{1.6}
\def\rad{0.8}

\coordinate (r) at (45:\rad);
\coordinate (g) at (135:\rad);
\coordinate (b) at (225:\rad);
\coordinate (m) at (315:\rad);

\coordinate (v1) at (r);
\coordinate (v2) at (g);
\coordinate (v3) at (b);
\coordinate (v4) at (m);

\draw[_edge] (v1)--(v2);
\draw[_edge] (v2)--(v3);
\draw[_edge] (v3)--(v4);
\draw[_edge] (v4)--(v1);

\coordinate (R) at (45:\Rad);
\coordinate (G) at (135:\Rad);
\coordinate (B) at (225:\Rad);
\coordinate (M) at (315:\Rad);

\coordinate (V1) at (R);
\coordinate (V2) at (G);
\coordinate (V3) at (B);
\coordinate (V4) at (M);

\draw[_edge] (V1)--(V2);
\draw[_edge] (V2)--(V3);
\draw[_edge] (V3)--(V4);
\draw[_edge] (V4)--(V1);

\draw[_edge] (V1)--(v1);
\draw[_edge] (V2)--(v2);
\draw[_edge] (V3)--(v3);
\draw[_edge] (V4)--(v4);

\node[_vertex] (u1) at (v1) {};
\node[_vertex] (u2) at (v2) {};
\node[_vertex] (u3) at (v3) {};
\node[_vertex] (u4) at (v4) {};
\node[_vertex] (U1) at (V1) {};
\node[_vertex] (U2) at (V2) {};
\node[_vertex] (U3) at (V3) {};
\node[_vertex] (U4) at (V4) {};
\end{scope}

\begin{scope}[xshift=3cm,yshift=3cm]
\def\Rad{1.5}
\def\rad{0.36}

\coordinate (r) at (270:\rad);
\coordinate (g) at (150:\rad);
\coordinate (b) at (30:\rad);

\coordinate (v1) at (r);
\coordinate (v2) at (g);
\coordinate (v3) at (b);

\draw[_edge] (v1)--(v2);
\draw[_edge] (v2)--(v3);
\draw[_edge] (v3)--(v1);

\coordinate (R) at (90:\Rad);
\coordinate (G) at (330:\Rad);
\coordinate (B) at (210:\Rad);

\coordinate (V1) at (R);
\coordinate (V2) at (G);
\coordinate (V3) at (B);

\draw[_edge] (V1)--(V2);
\draw[_edge] (V2)--(V3);
\draw[_edge] (V3)--(V1);

\draw[_edge] (V1)--(v2);
\draw[_edge] (V1)--(v3);
\draw[_edge] (V2)--(v1);
\draw[_edge] (V2)--(v3);
\draw[_edge] (V3)--(v1);
\draw[_edge] (V3)--(v2);

\node[_vertex] (u1) at (v1) {};
\node[_vertex] (u2) at (v2) {};
\node[_vertex] (u3) at (v3) {};
\node[_vertex] (U1) at (V1) {};
\node[_vertex] (U2) at (V2) {};
\node[_vertex] (U3) at (V3) {};
\end{scope}

\begin{scope}[xshift=-1.5cm,yshift=0cm]
\def\rad{0.4}
\def\Rad{0.72}
\def\RAd{0.9}
\def\RAD{1.3}
\foreach \angle/\i in {18/1,90/2,162/3,234/4,306/5}{
	\coordinate (v\i) at (\angle:\RAD);
}

\foreach \angle/\i in {18/6,90/7,162/8,234/9,306/10}{
	\coordinate (v\i) at (\angle:\RAd);
}

\foreach \angle/\i in {198/11,270/12,342/13,54/14,126/15}{
	\coordinate (v\i) at (\angle:\Rad);
}

\foreach \angle/\i in {198/16,270/17,342/18,54/19,126/20}{
	\coordinate (v\i) at (\angle:\rad);
}

\node[_vertex] (u1) at (v1) {};
\node[_vertex] (u2) at (v2) {};
\node[_vertex] (u3) at (v3) {};
\node[_vertex] (u4) at (v4) {};
\node[_vertex] (u5) at (v5) {};
\node[_vertex] (u6) at (v6) {};
\node[_vertex] (u7) at (v7) {};
\node[_vertex] (u8) at (v8) {};
\node[_vertex] (u9) at (v9) {};
\node[_vertex] (u10) at (v10) {};
\node[_vertex] (u11) at (v11) {};
\node[_vertex] (u12) at (v12) {};
\node[_vertex] (u13) at (v13) {};
\node[_vertex] (u14) at (v14) {};
\node[_vertex] (u15) at (v15) {};
\node[_vertex] (u16) at (v16) {};
\node[_vertex] (u17) at (v17) {};
\node[_vertex] (u18) at (v18) {};
\node[_vertex] (u19) at (v19) {};
\node[_vertex] (u20) at (v20) {};

\draw[_edge] (u1)--(u2);
\draw[_edge] (u2)--(u3);
\draw[_edge] (u3)--(u4);
\draw[_edge] (u4)--(u5);
\draw[_edge] (u5)--(u1);

\draw[_edge] (u1)--(u6);
\draw[_edge] (u2)--(u7);
\draw[_edge] (u3)--(u8);
\draw[_edge] (u4)--(u9);
\draw[_edge] (u5)--(u10);

\draw[_edge] (u14)--(u6);
\draw[_edge] (u13)--(u6);
\draw[_edge] (u15)--(u7);
\draw[_edge] (u14)--(u7);
\draw[_edge] (u11)--(u8);
\draw[_edge] (u15)--(u8);
\draw[_edge] (u12)--(u9);
\draw[_edge] (u11)--(u9);
\draw[_edge] (u13)--(u10);
\draw[_edge] (u12)--(u10);

\draw[_edge] (u11)--(u16);
\draw[_edge] (u12)--(u17);
\draw[_edge] (u13)--(u18);
\draw[_edge] (u14)--(u19);
\draw[_edge] (u15)--(u20);

\draw[_edge] (u17)--(u16);
\draw[_edge] (u18)--(u17);
\draw[_edge] (u19)--(u18);
\draw[_edge] (u20)--(u19);
\draw[_edge] (u16)--(u20);

\end{scope}

\begin{scope}[xshift=1.5cm,yshift=-3mm]
\def\Rad{1.5}
\def\RAd{0.5}
\def\rad{0.2}

\coordinate (r) at (270:\rad);
\coordinate (g) at (150:\rad);
\coordinate (b) at (30:\rad);

\coordinate (v1) at (r);
\coordinate (v2) at (g);
\coordinate (v3) at (b);

\draw[_edge] (v1)--(v2);
\draw[_edge] (v2)--(v3);
\draw[_edge] (v3)--(v1);

\coordinate (R) at (90:\Rad);
\coordinate (G) at (330:\Rad);
\coordinate (B) at (210:\Rad);

\coordinate (V1) at (R);
\coordinate (V2) at (G);
\coordinate (V3) at (B);

\draw[_edge] (V1)--(V2);
\draw[_edge] (V2)--(V3);
\draw[_edge] (V3)--(V1);

\foreach \angle/\i in {90/1, 150/2, 210/3, 270/4, 330/5, 30/6}{
	\coordinate (w\i) at (\angle:\RAd);
}

\draw[_edge] (w1)--(w2);
\draw[_edge] (w2)--(w3);
\draw[_edge] (w3)--(w4);
\draw[_edge] (w4)--(w5);
\draw[_edge] (w5)--(w6);
\draw[_edge] (w6)--(w1);

\draw[_edge] (v1)--(w5);
\draw[_edge] (v2)--(w1);
\draw[_edge] (v3)--(w1);
\draw[_edge] (v1)--(w4);
\draw[_edge] (v2)--(w2);
\draw[_edge] (v3)--(w5);
\draw[_edge] (v1)--(w3);
\draw[_edge] (v2)--(w3);
\draw[_edge] (v3)--(w6);

\draw[_edge] (V1)--(w1);
\draw[_edge] (V2)--(w4);
\draw[_edge] (V3)--(w2);
\draw[_edge] (V1)--(w2);
\draw[_edge] (V2)--(w5);
\draw[_edge] (V3)--(w3);
\draw[_edge] (V1)--(w6);
\draw[_edge] (V2)--(w6);
\draw[_edge] (V3)--(w4);

\node[_vertex] (u1) at (v1) {};
\node[_vertex] (u2) at (v2) {};
\node[_vertex] (u3) at (v3) {};
\node[_vertex] (U1) at (V1) {};
\node[_vertex] (U2) at (V2) {};
\node[_vertex] (U3) at (V3) {};
\node[_vertex] (W1) at (w1) {};
\node[_vertex] (W2) at (w2) {};
\node[_vertex] (W3) at (w3) {};
\node[_vertex] (W4) at (w4) {};
\node[_vertex] (W5) at (w5) {};
\node[_vertex] (W6) at (w6) {};

\end{scope}

\end{tikzpicture}

%% file: triangles.tex
\begin{tikzpicture} [scale=0.8,
  _vertex/.style ={circle,draw=black, fill=black,inner sep=0.5pt},
  r_vertex/.style={circle,draw=red,  fill=red,inner sep=0.5pt},
  g_vertex/.style={circle,draw=green,fill=green,inner sep=0.5pt},
  b_vertex/.style={circle,draw=blue, fill=blue,inner sep=0.5pt},
  _edge/.style={black,line width=0.2pt},
  r_edge/.style={red,line width=0.3pt},
  g_edge/.style={green,line width=0.3pt},
  b_edge/.style={blue,line width=0.3pt},
 every edge/.style={draw=black,line width=0.3pt},
t_edge/.style={cap=round, ultra thick}]

\begin{scope}[xshift=3cm]
\coordinate (O) at (0,0);
\coordinate (P) at (2,0);
\draw[_edge] (O)--(P);
\foreach \i in {-2,-1,1,2}{
	\coordinate (v\i) at (1,\i);
	\draw[_edge] (v\i)--(O);
	\draw[_edge] (v\i)--(P);
	\node[_vertex] (u\i) at (v\i) {};
}
\node[_vertex] at (O) {};
\node[_vertex] at (P) {};
\end{scope}

\begin{scope}[xshift=0cm]
\def\RAD{1.3}
\foreach \angle/\i in {90/1, 150/2, 210/3, 270/4, 330/5, 30/6}{
	\coordinate (v\i) at (\angle:\RAD);
}
\coordinate (v0) at (0,0);

\foreach \i/\j in {1/2,3/4,5/6}{
\draw[_edge] (v\i)--(v\j);
\draw[_edge] (v\i)--(v0);
\draw[_edge] (v\j)--(v0);
}

\foreach \i/\j in {0,1,2,3,4,5,6}{
\node[_vertex] (u\i) at (v\i) {};
}
\end{scope}

\end{tikzpicture}

%% file: wheels.tex
\usetikzlibrary{decorations,decorations.markings,decorations.text,calc,arrows}

\begin{tikzpicture} [scale=0.9,
  _vertex/.style ={circle,draw=black, fill=black,inner sep=0.5pt},
  r_vertex/.style={circle,draw=red,  fill=red,inner sep=0.5pt},
  g_vertex/.style={circle,draw=green,fill=green,inner sep=0.5pt},
  b_vertex/.style={circle,draw=blue, fill=blue,inner sep=0.5pt},
  _edge/.style={black,line width=0.2pt},
  r_edge/.style={red,line width=0.3pt},
  g_edge/.style={green,line width=0.3pt},
  b_edge/.style={blue,line width=0.3pt},
 every edge/.style={draw=black,line width=0.3pt},
t_edge/.style={cap=round, ultra thick}]

\begin{scope}[xshift=-4cm]
\def\RAD{1.3}
\foreach \angle/\i in {18/1,90/2,162/3,234/4,306/5}{
	\coordinate (v\i) at (\angle:\RAD);
}
\coordinate (v0) at (0,0);

\foreach \i/\j in {1/2,2/3,3/4,4/5,5/1}{
\draw[_edge] (v\i)--(v\j);
\draw[_edge] (v\i)--(v0);
}

\foreach \i/\j in {0,1,2,3,4,5}{
\node[_vertex] (u\i) at (v\i) {};
}

\end{scope}

\begin{scope}[xshift=0cm]
\def\RAD{1.3}
\foreach \angle/\i in {90/1, 150/2, 210/3, 270/4, 330/5, 30/6}{
	\coordinate (v\i) at (\angle:\RAD);
}
\coordinate (v0) at (0,0);

\foreach \i/\j in {1/2,2/3,3/4,4/5,5/6,6/1}{
\draw[_edge] (v\i)--(v\j);
\draw[_edge] (v\i)--(v0);
}

\foreach \i/\j in {0,1,2,3,4,5,6}{
\node[_vertex] (u\i) at (v\i) {};
}

\end{scope}

\begin{scope}[xshift=4cm]
\def\RAD{1.3}
\foreach \angle/\i in {90/1, 141/2, 192/3, 243/4, 294/5, 345/6, 36/7}{
	\coordinate (v\i) at (\angle:\RAD);
}
\coordinate (v0) at (0,0);

\foreach \i/\j in {1/2,2/3,3/4,4/5,5/6,6/7,7/1}{
\draw[_edge] (v\i)--(v\j);
\draw[_edge] (v\i)--(v0);
}

\foreach \i/\j in {0,1,2,3,4,5,6,7}{
\node[_vertex] (u\i) at (v\i) {};
}

\end{scope}

\end{tikzpicture}

%% file: pinata.tex
\usetikzlibrary{decorations,decorations.markings,decorations.text,calc,arrows}

\begin{tikzpicture} [scale=1,
  _vertex/.style ={circle,draw=black, fill=black,inner sep=0.5pt},
  r_vertex/.style={circle,draw=red,  fill=red,inner sep=0.5pt},
  g_vertex/.style={circle,draw=green,fill=green,inner sep=0.5pt},
  b_vertex/.style={circle,draw=blue, fill=blue,inner sep=0.5pt},
  _edge/.style={black,line width=0.2pt},
  r_edge/.style={red,line width=0.3pt},
  g_edge/.style={green,line width=0.3pt},
  b_edge/.style={blue,line width=0.3pt},
 every edge/.style={draw=black,line width=0.3pt},
t_edge/.style={cap=round, ultra thick}]

\begin{scope}[xshift=3.5cm]
\def\rad{1.5}
\def\Rad{0.9}

\foreach \angle/\i in {90/1,135/2,180/3,225/4,270/5,315/6,0/7,45/8}{
   \coordinate (v\i) at (\angle:\Rad);

}
\coordinate (v0) at (0,0);
\foreach \angle/\i in {112.5/1,157.5/2,202.5/3,247.5/4,292.5/5,337.5/6,22.5/7,67.5/8}{
   \coordinate (u\i) at (\angle:\rad);

}
\foreach \i/\j in {8/1,1/2,2/3,3/4,4/5,5/6,6/7,7/8}{
   \draw[_edge] (u\i) -- (u\j);
\draw[_edge] (v\i) -- (v\j);
\draw[_edge] (v\i) -- (v0);
\draw[_edge] (u\i) -- (v\j);
\draw[_edge] (u\j) -- (v\j);

}
\foreach \i in {1,2,3,4,5,6,7,8}{
   \node[_vertex] (V\i) at (v\i){};
\node[_vertex] (U\i) at (u\i){};

}
\node[_vertex] (V0) at (v0){};

\end{scope}

\begin{scope}[xshift=-3.5cm]
\def\rad{1.5}
\def\Rad{0.9}

\foreach \angle/\i in {90/1,150/2,210/3,270/4,330/5,30/6}{
   \coordinate (v\i) at (\angle:\Rad);

}
\coordinate (v0) at (0,0);
\foreach \angle/\i in {120.0/1,180.0/2,240.0/3,300.0/4,0.0/5,60.0/6}{
   \coordinate (u\i) at (\angle:\rad);

}
\foreach \i/\j in {6/1,1/2,2/3,3/4,4/5,5/6}{
   \draw[_edge] (u\i) -- (u\j);
\draw[_edge] (v\i) -- (v\j);
\draw[_edge] (v\i) -- (v0);
\draw[_edge] (u\i) -- (v\j);
\draw[_edge] (u\j) -- (v\j);

}
\foreach \i in {1,2,3,4,5,6}{
   \node[_vertex] (V\i) at (v\i){};
\node[_vertex] (U\i) at (u\i){};

}
\node[_vertex] (V0) at (v0){};
\end{scope}

\begin{scope}[xshift=0cm]
\def\rad{1.5}
\def\Rad{0.9}

\foreach \angle/\i in {90.0/1,141.428571429/2,192.857142857/3,244.285714286/4,295.714285714/5,347.142857143/6,38.5714285714/7}{
   \coordinate (v\i) at (\angle:\Rad);

}
\coordinate (v0) at (0,0);
\foreach \angle/\i in {115.714285714/1,167.142857143/2,218.571428571/3,270.0/4,321.428571429/5,12.8571428571/6,64.2857142857/7}{
   \coordinate (u\i) at (\angle:\rad);

}
\foreach \i/\j in {7/1,1/2,2/3,3/4,4/5,5/6,6/7}{
   \draw[_edge] (u\i) -- (u\j);
\draw[_edge] (v\i) -- (v\j);
\draw[_edge] (v\i) -- (v0);
\draw[_edge] (u\i) -- (v\j);
\draw[_edge] (u\j) -- (v\j);

}
\foreach \i in {1,2,3,4,5,6,7}{
   \node[_vertex] (V\i) at (v\i){};
\node[_vertex] (U\i) at (u\i){};

}
\node[_vertex] (V0) at (v0){};

\end{scope}

\end{tikzpicture}

%% file: cubdividido.tex
\usetikzlibrary{decorations,decorations.markings,decorations.text,calc,arrows}

\begin{tikzpicture} [scale=0.7,
  _vertex/.style ={circle,draw=black, fill=black,inner sep=0.5pt},
  r_vertex/.style={circle,draw=red,  fill=red,inner sep=0.5pt},
  g_vertex/.style={circle,draw=green,fill=green,inner sep=0.5pt},
  b_vertex/.style={circle,draw=blue, fill=blue,inner sep=0.5pt},
  _edge/.style={black,line width=0.2pt},
  r_edge/.style={red,line width=0.3pt},
  g_edge/.style={green,line width=0.3pt},
  b_edge/.style={blue,line width=0.3pt},
 every edge/.style={draw=black,line width=0.3pt},
t_edge/.style={cap=round, ultra thick}]

\def\rad{1.5}

\begin{scope}[xshift=3cm]
\coordinate (u0) at (0,0);
\foreach \angle/\i in {90.00/1,126.00/2,162.00/3,198.00/4,234.00/5,270.00/6,306.00/7,342.00/8,738.00/9,774.00/10}
{
   \coordinate (u\i) at (\angle:\rad);
}
\foreach \i/\j in {1/2,2/3,3/4,4/5,5/6,6/7,7/8,8/9,9/10,10/1}
{
   \draw[_edge] (u\i) -- (u\j);
}
\foreach \i in {1,3,5,7,9}
{
  \draw[_edge] (u\i) -- (u0);
}
\foreach \i in {1,2,3,4,5,6,7,8,9,10}
{
   \node[_vertex] (U\i) at (u\i) {};
}
\node[_vertex] (U0) at (u0) {};
\end{scope}

\begin{scope}[xshift=-3cm]
\def\rad{1.5}
\def\Rad{0.9}

\foreach \angle/\i in {90/1,135/2,180/3,225/4,270/5,315/6,0/7,45/8}{
   \coordinate (v\i) at (\angle:\rad);
}
\coordinate (v0) at (0,0);

\foreach \i in {1,3,5,7}{
\draw[_edge] (v\i) -- (v0);
}

\foreach \i/\j in {8/1,1/2,2/3,3/4,4/5,5/6,6/7,7/8}{
\draw[_edge] (v\i) -- (v\j);
}
\foreach \i in {1,2,3,4,5,6,7,8}{
   \node[_vertex] (V\i) at (v\i){};
\node[_vertex] (U\i) at (u\i){};

}
\node[_vertex] (V0) at (v0){};
\end{scope}

\end{tikzpicture}

%% file: pentapinatas.tex
\begin{tikzpicture} [scale=0.8,
  _vertex/.style ={circle,draw=black, fill=black,inner sep=0.5pt},
  _edge/.style={black,line width=0.2pt},
rotate=180]

\begin{scope}[xshift=4cm]
\def\ra{0.5}
\def\rb{1}
\def\rc{1.6}
\def\rd{2}
\coordinate (u0) at (0,0);
\foreach \angle/\i in {90.00/1,180.00/2,270.00/3,0.00/4}
{
   \coordinate (v\i) at (\angle:\ra);
}
\foreach \angle/\i in {67.50/1,112.50/2,157.50/3,202.50/4,247.50/5,292.50/6,337.50/7,22.50/8}
{
   \coordinate (w\i) at (\angle:\rb);
}
\foreach \angle/\i in {67.50/1,112.50/2,157.50/3,202.50/4,247.50/5,292.50/6,337.50/7,22.50/8}
{
   \coordinate (x\i) at (\angle:\rc);
}
\foreach \angle/\i in {135.00/1,225.00/2,315.00/3,45.00/4}
{
   \coordinate (y\i) at (\angle:\rd);
}
\foreach \i in {1,2,...,4}{
    \pgfmathsetmacro\j{2*\i-1}
    \pgfmathsetmacro\k{2*\i}
    \draw[_edge] (v\i) -- (u0);
    \draw[_edge] (v\i) -- (w\j);
    \draw[_edge] (v\i) -- (w\k);
    \draw[_edge] (x\j) -- (w\j);
    \draw[_edge] (x\k) -- (w\k);
    \draw[_edge] (x\j) -- (x\k);
    \draw[_edge] (y\i) -- (x\k);
}
\foreach \i/\k in {1/3,2/5,3/7,4/1}{
    \pgfmathsetmacro\j{2*\i}
    \draw[_edge] (w\j) -- (w\k);
    \draw[_edge] (y\i) -- (x\k);
}
\foreach \i in {1,2,3,4}
{
   \node[_vertex] (V\i) at (v\i) {};
}
\foreach \i in {1,2,3,4,5,6,7,8}
{
   \node[_vertex] (W\i) at (w\i) {};
}
\foreach \i in {1,2,3,4,5,6,7,8}
{
   \node[_vertex] (X\i) at (x\i) {};
}
\foreach \i in {1,2,3,4}
{
   \node[_vertex] (Y\i) at (y\i) {};
}
\end{scope}

\begin{scope}[xshift=0cm]
\def\ra{0.5}
\def\rb{1}
\def\rc{1.6}
\def\rd{2}
\coordinate (u0) at (0,0);
\foreach \angle/\i in {90.00/1,162.00/2,234.00/3,306.00/4,18.00/5}
{
   \coordinate (v\i) at (\angle:\ra);
}
\foreach \angle/\i in {72.00/1,108.00/2,144.00/3,180.00/4,216.00/5,252.00/6,288.00/7,324.00/8,0.00/9,36.00/10}
{
   \coordinate (w\i) at (\angle:\rb);
}
\foreach \angle/\i in {72.00/1,108.00/2,144.00/3,180.00/4,216.00/5,252.00/6,288.00/7,324.00/8,0.00/9,36.00/10}
{
   \coordinate (x\i) at (\angle:\rc);
}
\foreach \angle/\i in {126.00/1,198.00/2,270.00/3,342.00/4,54.00/5}
{
   \coordinate (y\i) at (\angle:\rd);
}
\foreach \i in {1,2,...,5}{
    \pgfmathsetmacro\j{2*\i-1}
    \pgfmathsetmacro\k{2*\i}
    \draw[_edge] (v\i) -- (u0);
    \draw[_edge] (v\i) -- (w\j);
    \draw[_edge] (v\i) -- (w\k);
    \draw[_edge] (x\j) -- (w\j);
    \draw[_edge] (x\k) -- (w\k);
    \draw[_edge] (x\j) -- (x\k);
    \draw[_edge] (y\i) -- (x\k);
}
\foreach \i/\k in {1/3,2/5,3/7,4/9,5/1}{
    \pgfmathsetmacro\j{2*\i}
    \draw[_edge] (w\j) -- (w\k);
    \draw[_edge] (y\i) -- (x\k);
}
\foreach \i in {1,2,3,4,5}
{
   \node[_vertex] (V\i) at (v\i) {};
}
\foreach \i in {1,2,3,4,5,6,7,8,9,10}
{
   \node[_vertex] (W\i) at (w\i) {};
}
\foreach \i in {1,2,3,4,5,6,7,8,9,10}
{
   \node[_vertex] (X\i) at (x\i) {};
}
\foreach \i in {1,2,3,4,5}
{
   \node[_vertex] (Y\i) at (y\i) {};
}
\end{scope}

\begin{scope}[xshift=-4cm]
\def\ra{0.5}
\def\rb{1}
\def\rc{1.6}
\def\rd{2}

\coordinate (u0) at (0,0);
\foreach \angle/\i in {90.00/1,135.00/2,180.00/3,225.00/4,270.00/5,315.00/6,0.00/7,45.00/8}
{
   \coordinate (v\i) at (\angle:\ra);
}
\foreach \angle/\i in {78.75/1,101.25/2,123.75/3,146.25/4,168.75/5,191.25/6,213.75/7,236.25/8,258.75/9,281.25/10,303.75/11,326.25/12,348.75/13,11.25/14,33.75/15,56.25/16}
{
   \coordinate (w\i) at (\angle:\rb);
}
\foreach \angle/\i in {78.75/1,101.25/2,123.75/3,146.25/4,168.75/5,191.25/6,213.75/7,236.25/8,258.75/9,281.25/10,303.75/11,326.25/12,348.75/13,11.25/14,33.75/15,56.25/16}
{
   \coordinate (x\i) at (\angle:\rc);
}
\foreach \angle/\i in {112.50/1,157.50/2,202.50/3,247.50/4,292.50/5,337.50/6,22.50/7,67.50/8}
{
   \coordinate (y\i) at (\angle:\rd);
}
\foreach \i in {1,2,...,8}{
    \pgfmathsetmacro\j{2*\i-1}
    \pgfmathsetmacro\k{2*\i}
    \draw[_edge] (v\i) -- (u0);
    \draw[_edge] (v\i) -- (w\j);
    \draw[_edge] (v\i) -- (w\k);
    \draw[_edge] (x\j) -- (w\j);
    \draw[_edge] (x\k) -- (w\k);
    \draw[_edge] (x\j) -- (x\k);
    \draw[_edge] (y\i) -- (x\k);
}
\foreach \i/\k in {1/3,2/5,3/7,4/9,5/11,6/13,7/15,8/1}{
    \pgfmathsetmacro\j{2*\i}
    \draw[_edge] (w\j) -- (w\k);
    \draw[_edge] (y\i) -- (x\k);
}
\foreach \i in {1,2,3,4,5,6,7,8}
{
   \node[_vertex] (V\i) at (v\i) {};
}
\foreach \i in {1,2,3,4,5,6,7,8,9,10,11,12,13,14,15,16}
{
   \node[_vertex] (W\i) at (w\i) {};
}
\foreach \i in {1,2,3,4,5,6,7,8,9,10,11,12,13,14,15,16}
{
   \node[_vertex] (X\i) at (x\i) {};
}
\foreach \i in {1,2,3,4,5,6,7,8}
{
   \node[_vertex] (Y\i) at (y\i) {};
}
\end{scope}
\end{tikzpicture}